\documentclass{birkjour}
\usepackage{graphics,xcolor}
\newtheorem{thm}{Theorem}[section]
\newtheorem{cor}[thm]{Corollary}
\newtheorem{lem}[thm]{Lemma}
\newtheorem{prop}[thm]{Proposition}
\theoremstyle{definition}

\theoremstyle{remark}
\newtheorem{rem}[thm]{Remark}
\numberwithin{equation}{section}

\newcommand{\al}{\alpha}

\newcommand{\la}{\lambda}

\newcommand{\cA}{\mathcal{A}}
\newcommand{\cB}{\mathcal{B}}
\newcommand{\cC}{\mathcal{C}}
\newcommand{\cD}{\mathcal{D}}

\newcommand{\cG}{\mathcal{G}}
\newcommand{\cL}{\mathcal{L}}
\newcommand{\cH}{\mathcal H}

\newcommand{\cM}{\mathcal M}
\newcommand{\cP}{\mathcal P}
\newcommand{\cW}{\mathcal W}
\newcommand{\f}{{\boldsymbol{f}}}
\newcommand{\g}{{\boldsymbol{g}}}
\newcommand{\h}{{\boldsymbol{h}}}
\newcommand{\m}{{\boldsymbol{m}}}
\newcommand{\y}{{\boldsymbol{y}}}
\newcommand{\x}{{\boldsymbol{x}}}
\newcommand{\bu}{{\boldsymbol{u}}}
\newcommand{\bxi}{{\boldsymbol{\xi}}}

\renewcommand{\L}{{\boldsymbol{L}}}
\renewcommand{\P}{{\boldsymbol{P}}}
\newcommand{\Y}{{\boldsymbol{Y}}}

\newcommand{\wh}{\widehat}
\newcommand{\ol}{\overline}
\newcommand{\hf}{\widehat f}
\newcommand{\hg}{\widehat g}
\newcommand{\tr}{{\rm tr}}
\renewcommand{\i}{{\rm i}}
\newcommand{\e}{{\rm e}}

\newcommand{\mS}{\mathcal S}
\newcommand{\mT}{\mathcal T}

\DeclareMathOperator{\diag}{diag}

\DeclareMathOperator{\dom}{dom}
\DeclareMathOperator{\ran}{ran}

\renewcommand\Im{{\rm Im}\,}

\definecolor{ao}{rgb}{0,0.5,0}
\definecolor{armygreen}{rgb}{0.29,0.33,0.13}
\definecolor{auburn}{rgb}{0.43,0.21,0.1}

\newcommand{\ct}[1]{\textcolor{black}{#1}}

\newcommand\C{\mathbb C}
\newcommand\R{\mathbb R}
\newcommand\Z{\mathbb Z}
\newcommand\N{\mathbb N}
\renewcommand\d{{\rm d}}

\begin{document}

\title[Dirac-Krein systems on star graphs]{Dirac-Krein systems on star \vspace{-2mm} graphs}%
\author[V.\ Adamyan]{V.\ Adamyan}%
\address{
Odessa National I.I.\ Mechnikov University\\
Dvoryanska 2\\ 
65082 Odessa\\ 
Ukraine}%
\email{vadamyan@onu.edu.ua}%
\author[H.\ Langer]{H.\ Langer}%
\address{
Institut f\"ur Analysis und Scientific Computing\\
Technische Universit\"at Wien\\ 
Wiedner Hauptstrasse 8-10\\
1040 Vienna\\
Austria}%
\email{heinz.langer@tuwien.ac.at}%
\author[C.\ Tretter]{C.\ Tretter}%
\address{
\parbox[t]{4.7cm}{Mathematisches Institut \quad  {\rm\&}\\ 
Universit\"at Bern\\
Sidlerstr.\ 5 \\ 
3012 Bern\\
Switzerland\vspace{2mm}} 
\parbox[t]{6cm}{
Matematiska institutionen\\
Stockholms universitet\\
106 91 Stockholm\\ 
Sweden}}
\email{tretter@math.unibe.ch} %
\author[M.\ Winklmeier]{M.\ Winklmeier}%
\address{
Department of Mathematics\\
Universidad de los Andes\\
111711 Bogot\'a\\
Colombia}%
\email{mwinklme@uniandes.edu.co}%
\date{\today}

\dedicatory{Dedicated to Prof.\ Reinhard Mennicken on the occasion of his 80th \vspace{-2mm} birthday}%
\begin{abstract}
We study the spectrum of a self-adjoint Dirac-Krein operator with potential on a compact star graph $\mathcal G$ with a finite number $n$ of edges. 
This operator is defined by a Dirac-Krein differential expression with summable matrix potentials on each edge,
by self-adjoint boundary conditions at the outer vertices, and by a self-adjoint matching condition at the common central vertex of $\cG$.
Special attention is paid to Robin matching conditions with para\-meter $\tau \in\R\cup\{\infty\}$.
Choosing the decoupled operator with Dirichlet condition at the central vertex as a reference operator, 
we derive Krein's resolvent formula, introduce corresponding Weyl--Titchmarsh functions, study the multiplicities, dependence on $\tau$, and interlacing properties of the eigenvalues, and prove a trace formula.
Moreover, we show that, asymptotically for $R\to \infty$, the difference of the number of eigenvalues in the intervals $[0,R)$ and $[-R,0)$ deviates from some integer $\kappa_0$, which we call dislocation index, at most by $n+2$.
\vspace{-1mm}
\end{abstract}

\maketitle
\section{Introduction}

While there exists a vast literature on Schr\"odinger operators on metric graphs and their spectral properties (see e.g.\ \cite{MR2459860},  \cite{MR3013208}, and the numerous lists of references therein), 
there are only a few papers on Dirac operators on metric graphs so far (see e.g.\ \cite{MR1050469}, \cite{MR1965289}, \cite{MR2459874}, \cite{MR2533873}, \cite{MR2798800}, \cite{MR2968647}, and \cite{MR2989218}).  
This may be due to the fact that classical techniques for semi-bounded operators do not apply to Dirac operators
and so mostly only special cases such as free Dirac operators or particular boundary and/or matching conditions are treated. 

On the other hand, interest in and need for mathematical results on Dirac operators has rapidly grown in the last years
due to novel applications and experiments. They include models for electronic properties of graphene (see e.g.\ \cite{wolf14}), propagation of electromagnetic waves in graphene-like photonic crystals 
(see e.g.\ \cite{PhysRevB.91.035411}), ultracold matter in optical lattices and  
``proof-of-principle quantum simulation of the one-dimensional Dirac equation using a single trapped ion-set" 
to confirm amazing features of relativistic quantum motion like Zitterbewegung and Klein paradox (see \cite{PhysRevA.84.033601}, \cite{nature}).

In this paper, we establish some fundamental spectral properties of Dirac operators on star graphs.
Since physically relevant vertex couplings are of local nature, star graphs play a role in quantum computing \ct{and vibrations of networks} 
(see e.g.\ \cite{MR3304876}, \ct{\cite{MR3334518}}) and may be viewed as building blocks for more complicated graph geometries \ct{(see e.g.\ 
\cite{MR3400894})}; \ct{the latter} also appear as shrinking limits of thin branched manifolds (see e.g.\ \cite{MR3229166},~\cite{MR3073163}).

Here we study Dirac-Krein operators with potential on a compact star graph $\cG$ with~$n$ edges $e_j$  defined by differential \vspace{-1.75mm}  expressions
\[
-\begin{pmatrix}
  0 & -1 \\
  1 & 0
 \end{pmatrix}\frac{\d}{\d x}+
 \begin{pmatrix}
 p_j(x) & q_j(x) \\
 q_j(x) & -p_j(x) \\
 \end{pmatrix},\quad x\in e_j, \ j=1,2,\dots,n,
\vspace{-1.75mm}
\]
on the edges with real-valued $L^1$-functions $p_j$,~$q_j$, arbitrary symmetric boundary conditions at the outer vertices and a self-adjoint matching condition at the common central vertex. The main results may be summarized as follows.
First we establish Krein's resolvent formula describing the resolvents of all self-adjoint ex\-tensions $\mT$ of the corresponding minimal Dirac-Krein operator on $\cG$ with respect to the fixed decoupled extension $\mT_0$ given by Dirichlet conditions at the central vertex.  
The other two main results concern extensions $\mT_\tau$ given by Robin matching conditions at the central vertex with parameter $\tau\in\R\cup\{\infty\}$.
We establish a trace formula and we prove that the set of eigenvalues splits into two parts, simple eigenvalues that depend on the Robin parameter $\tau$ monotonically and possibly non-simple eigenvalues that are independent of $\tau$. Moreover, 
we show that the eigenvalues for different Robin parameters $\tau_1$ and $\tau_2$ interlace. Thirdly we prove that, asymptotically for $R\to \infty$, the difference $d_R(\mT_\tau)$ of the number of eigenvalues of $\mT_\tau$ in intervals $[0,R)$ 
and $[-R,0)$ differs from some fixed integer $\kappa_0$, which we call dislocation index, at most by~$n+2$; an even more precise estimate is given in terms of the boundary coefficients at the outer vertices.

A short synopsis of the paper is as follows. In Section \ref{sec:symmvertex} we introduce the $L^2$-spaces for the  graph $\cG$  and its edges $e_j=[v,v_j]$, the boundary conditions at the outer vertices $v_j$, 
the general form of the matching condition at the central vertex $v$, and the corresponding symmetric and self-adjoint operators. The boundary conditions $\cos\alpha_j\,f_j(v_j)+\sin \alpha_j \,\hf_j(v_j)=0$
at the outer vertices are considered to be fixed in this paper, whereas the matching condition at $v$ is varied;
here $f_j$, $\wh f_j$ are the two components of a vector function $\f_j = (f_j\, \wh f_j)^{\rm t}$ on an edge $e_j$.
Special attention is paid to matching condition of Robin type with parameter \vspace{-1mm}$\tau \in\R\cup \{\infty\}$,~i.e.\
\[
f_1(v)= f_2(v) = \cdots = f_n(v),\quad   f_1(v)= \tau \big( \wh f_1(v)+\wh f_2(v)+\cdots+\wh f_n(v) \big). 
\vspace{-1mm}
\]

In Section \ref{sec:CanonicalSystems} we summarize results for a Dirac-Krein operator  $\mT_j$ on a single edge $e_j$ (i.e.\ on an interval). Here the Weyl--Titchmarsh function $m_j$ is introduced and 
it is shown that it is a Nevanlinna function, i.e.\ it is holomorphic at least in 
$\C\setminus\R$, 
$m_j(\ol z)=\ol{m_j(z)}$, and $\Im m_j(z)/\Im z \ge 0$ for $z\in\C\setminus\R$.
In this section we also derive a trace formula, i.e.\ an analytic expression for the trace of the resolvent difference $(\mT_j-z_1)^{-1}-(\mT_j-z_2)^{-1}$ in two points $z_1$, $z_2 \in \rho(\mT_j)$.

The first main result of the paper is  Theorem \ref{thm:kres} in Section \ref{s4} where we describe the resolvents of the self-adjoint extensions 
$\mT_{\cA,\cB}$ of the minimal operator $\mS$ on $\cG$ through M.G.~Krein's resolvent formula in terms of the boundary matrices in
$\cA (f_j(v))_1^n + \cB (\wh f_j(v))_1^n =0$. It is the starting point for the study of the spectrum in the following sections. 
We also indicate how the Fourier expansion and the spectral representation for the operator $\mT_{\cA,\cB}$ follow (see Remark \ref{fourier}), 
but these questions are not considered in detail~here.

Starting from Section \ref{sec:tau}, we concentrate on Robin matching conditions at the central vertex. We show that the zeros of the Weyl-Titchmarsh \vspace{-1mm} function 
\[
   \m_\tau(z):= \dfrac 1\tau + \sum_{j=1}^n m_j(z), \quad z\in \C\setminus \R,
\vspace{-1mm}
\]
give rise to the simple eigenvalues of $\mT_\tau$, while possibly multiple eigenvalues occur if at least two $m_j$ have a common pole.
A corresponding result for symmetric relations on star graphs, but with $\tau=\infty$, was proved in \cite{MR3180877}.
We also study the monotonicity in $\tau$ of both types of eigenvalues, 
we prove an interlacing principle for the eigenvalues of the self-adjoint operators $\mT_\tau$ for two different values of $\tau$, and 
we derive a formula for the trace of the difference of the resolvents of the self-adjoint operators $\mT_\tau$ and $\mT_0$, i.e.\ for the perturbation determinant of $\mT_\tau$ with respect to $\mT_0$.

While 
most of the results of Sections \ref{sec:symmvertex}-\ref{sec:tau} remain valid -~with slight modifications~-  for more general canonical systems, 
this is not the case for the main theorem of Section \ref{sec:disloc} where we study the difference $d_R(\mT_\tau)$ of the number of eigenvalues of $\mT_\tau$ in intervals $[0,R)$ and $[-R,0)$. 
We prove that there exists a number $\kappa_0\in\N$, the dislocation index, such that, for $R>0$ sufficiently large,  $d_R(\mT_\tau)$ differs from $\kappa_0$ at most by $n+2$.
More precisely,  
\[
 -(n_\ge + 2 ) \le d_R(\mT_\tau) - \kappa_0 \le n_\le + 2
\]
where $n_\ge$ and $n_{\le}$ are the numbers of edges for which the parameters $\al_j$ in the boundary conditions at the outer vertices satisfy $\alpha_j \ge \frac \pi2$ and $\alpha_j \le \frac \pi2$, respectively.  
The proof is based on the asymptotic behavior of the eigenvalues of Dirac-Krein operators on an interval established e.g.~in \cite{LM14} (see also \cite{MR3340175}) and on the interlacing principle proved 
in Section \ref{sec:tau}. Finally, we derive an analytic formula for the dislocation index $\kappa_0$ of $\mT_\tau$ in Theorem \ref{finally!} which is based on  the trace formula for $T_j$  in Section~\ref{sec:CanonicalSystems}.

\section{Symmetric vertex conditions}%
\label{sec:symmvertex}

Consider a star graph $\mathcal G$ with central vertex $v$ and edges $ e_1,\dots,e_n$ 
whose endpoints are denoted by $v_1,\dots,v_n$. On each edge $e_j$ we fix the direction to 
be outgoing from the central vertex $v$, i.e.\ we identify $e_j$ with the interval $[v,v_j]$. We denote by $\L^2(e_j)=L^2(e_j,\C)\oplus L^2(e_j,\C)$ 
the Hilbert space of square-integrable $2$-vector functions $\f_j\!=\!\big( f_j\ \hf_j \big)^{\rm t}\!$ 
on $e_j$ with inner product
\[
(\f_j,\g_j)_{\L^2(e_j) }\!:=\!(f_j,g_j)\!+\!(\hf_j,\hg_j),
\quad \f_j,\g_j\in \L^2(e_j),
\]
and corresponding norm $\|\cdot\|_{\L^{2}(e_j)}$.
Then the vector space $\cH=\L^{2}(\cG):= \L^2(e_1)\oplus \L^2(e_2)\oplus\cdots\oplus \L^2(e_n)$
of functions $\f=\begin{pmatrix}\f_1 \ \f_2 \ \cdots \ \f_n \end{pmatrix}^{\rm t}=:\big(\f_j\big)_1^n$
becomes a Hilbert space if we equip it with inner product given by
\[
(\f,\g)_{\L^2(\cG)}\!:=\!\sum_{j=1}^n\left((f_j,g_j)\!+\!(\hf_j,\hg_j)\right),\quad
\f,\g\in \L^2(\cG),
\]
and corresponding norm $\|\cdot\|_{\L^{2}(\mathcal G)}$.
The variables on $e_j$ are denoted by $x_j$ and those on $\cG$ by $\x\!=\!(x_j)_1^n$, i.e.\ we identify $\f(\x) \!=\! ( \f_j(x_j ) )_1^n$ for $\x\!=\!(x_j)_1^n$ with $x_j\!\in\! e_j$; 
in the particular case $x_j=v$, $j=1,2$, $\dots,n$, we simply write $\f(v)$.

For fixed $j\in \{1,2,\dots,n\}$,  we consider a so-called Dirac-Krein differential expression
\begin{equation}\label{DK}
   H_j=-J\dfrac{\d}{\d x_j }+V_j \quad \text{on } \ e_j=[v,v_j]
\end{equation}
where
\begin{equation}
\label{JVj}
  J:=
   \begin{pmatrix}
      0 & -1 \\
      1 & 0
   \end{pmatrix},\quad V_j(x_j ):=
\begin{pmatrix}
 p_j(x_j ) & q_j(x_j ) \\
 q_j(x_j ) & -p_j(x_j )
 \end{pmatrix}, \ \ x_j \in e_j,
\end{equation}
with $p_j$,  $q_j \in L^1(e_j,\R)$, $j=1,2,\dots,n$. 

With the  differential expression  $H_j$ in \eqref{DK}, which is regular at both end-points, we associate various operators 
in $\L^2(e_j)$. At the outer endpoint $v_j$ of $e_j$ we always fix the self-adjoint boundary condition 
\begin{equation}
\label{bc_j}
  \cos\alpha_j\,f_j(v_j)+\sin \alpha_j \,\hf_j(v_j)=0 \quad \text{with } \ \alpha_j\in [0,\pi)
\end{equation}
and introduce the corresponding maximal domain
\[
 \cD_j := \{ \f_j\in \L^2(e_j) : \f_j\ \text{abs.\ cont.},\ H_j\f_j\in \L^2(e_j),\ \f_j \ \text{satisfies } \eqref{bc_j} \}.
\]
The operators $S_j$ and $T_j$ in $\L^2(e_j)$ are defined by different boundary conditions at the central vertex $v$,
\begin{equation}
\label{Tj}
\begin{array}{ll}
   \dom S_j := \{ \f_j\in \mathcal \cD_j : f_j(v) = \widehat f_j(v) = 0\}, \quad & S_j \f_j := H_j \f_j,
   \\[1mm]
   \dom T_j := \{ \f_j\in \mathcal \cD_j : f_j(v) = 0\}, \quad & T_j \f_j := H_j \f_j.
\end{array}
\end{equation}

Then $S_j$ is symmetric with defect index $(1,1)$ and adjoint 
\[
\dom S_j^* = \cD_j,  \quad S_j^* \f_j = H_j \f_j,
\]
and $T_j$ is a self-adjoint extension of $S_j$, 
\[
S_j\subset T_j = T_j^* \subset S_j^*.
\]

For $\f_j$, $\g_j\in\dom S^*_j$, using $J^*=-J$ we have
\begin{align*}
&(S^*_j\f_j,\g_j)-(\f_j,S^*_j\g_j)\\
&=\int_{e_j}\biggl(\biggl(-J\binom{f_j'(x_j )}{{\wh f_j}^{\,\prime}(x_j )},\binom{g_j(x_j )}{{\widehat{g_j}(x_j )}}\biggr)_{\!\!\C^2}
  -\left(J\binom{f_j(x_j )}{\hf_j(x_j )},\binom{g_j'(x_j )}{{\wh{g_j}'(x_j )}}\biggr)_{\!\!\C^2}\right)\,\d x_j \\
&=- \big( \hf_j(v)\ol{g_j(v)}-f_j(v)\ol{\widehat{g_j}(v)} \big) \\[1mm]
&=\left( J \f_j(v),\g_j(v)\right)_{\C^2}.
\end{align*}

On the  graph $\cG$ we study the self-adjoint extensions of the symmetric \vspace{-1mm}operator 
\begin{align}
\label{symmS}
\mS:=S_1\oplus S_2\oplus\cdots\oplus S_n
\end{align}
in $\L^2(\cG)$, i.e.\ the self-adjoint restrictions of the adjoint $\mS^*\!=\!S^*_1\oplus S^*_2\oplus\cdots\oplus S^*_n$.
To this end, it is more advantageous to reorder the $2n$ components of $\L^2(\cG)$. 
Thus,  for $\bu \!=\!\big(u_1\ \wh u_1\ u_2\ \wh u_2\, \cdots\,  u_n\, \wh u_n\big)^{\rm t}\!\in\!\mathbb C^{2n}$ 
we define $\bu ^\sharp \!\!\in\! \C^{2n}$ by
\begin{equation}
\label{perm}
\bu ^\sharp 
\!:=\!\big(u_1\ u_2\ \cdots\ u_n\ \wh u_1\ \wh u_2\ \cdots\ \wh u_n\big)^{\rm t} 
= \P \bu , 
\end{equation}
where $\P = \big( p_{ij} \big)_{i,j=1}^{2n} \in M_{2n}(\C)$ is the permutation matrix given by
\[
    p_{ij} \!:=\! 
    \begin{cases}  
    \begin{array}{ll} \!\!1, & j=2i-1, \\ \!\!0, & j \in \{1,2,\dots,2n\}\!\setminus\!\{2i-1\}, \end{array} & \hspace{-3mm} i=1,2,\dots,n,  \\
    \begin{array}{ll} \!\!1, & j=2(i-n), \\ \!\!0, & j \in \{1,2,\dots,2n\}\!\setminus\!\{2(i\!-\!n)\}, \end{array} & \hspace{-3mm} i=n\!+\!1,n\!+\!2,\dots,2n;
    \end{cases}
\]
note that $\P^{-1}=\P^{\rm t}$. Correspondingly, for $\f \!\in\! \L^2(\cG)$ we set
\begin{equation}
\label{fsharp}
   \f^\sharp (\x) \!:=\! \f(\x)^\sharp, \quad \x \!\in\! \cG.
\end{equation}
Further, we denote the components of $\bu^\sharp$ and $\f^\sharp$ with respect to the decompositions $\C^{2n}\!=\!\C^n\oplus \C^n$ and $\L^2(\cG) \!=\! \L^2(e_j)^n \oplus \L^2(e_j)^n$, 
respectively,~by
\begin{equation}
\label{perm1}
\bu ^\sharp \!=\!\begin{pmatrix}\bu ^\sharp_1\\ \bu _2^\sharp\end{pmatrix} \!:=\! \begin{pmatrix} (u_i)_1^n \\ (\wh u_i)_1^n \end{pmatrix}, \quad 
\f^\sharp \!=\! \begin{pmatrix}{\f}_1^\sharp\\ {\f}_2^\sharp \end{pmatrix} \!:=\!  \begin{pmatrix} (f_i)_1^n \\ (\wh f_i)_1^n \end{pmatrix}.
\end{equation}

With this notation, for $\f=(\f_j)_1^n$, $\g=(\g_j)_1^n \in \dom \mS^* = \bigoplus_1^n \dom S_j^*$, we~obtain
\begin{equation*}
\hspace{-4mm}
\begin{array}{rl}
(\mS^*\f,\g)\!-\!(\f,\mS^*\g)\!\!\!\!
&=\!\displaystyle \sum_{j=1}^n (S^*_j\f_j,\g_j)-\!\sum_{j=1}^n(\f_j,S^*_j\g_j)\!
=\!\displaystyle\sum_{j=1}^n\left( J \f_j(v),\g_j(v)\right)_{\C^2} \hspace{-5mm} \\
& = - \big( ( \f_2^\sharp(v) , \g_1^\sharp(v) )_{\C^n} - (\f_1^\sharp(v), \g_2^\sharp(v))_{\C^n} \big) \\[2mm]
&=\left({\bf J}_{2n}\f^\sharp(v),\g^\sharp(v)\right)_{\C^{2n}}
\end{array}
\end{equation*}
with 
\[
{\bf J}_{2n}:=\begin{pmatrix}0_n&-I_n\\I_n&0_n\end{pmatrix} = \P \,\diag ( \underbrace{J, \dots, J}_{n \ \text{times}}) \, \P^{\rm t},
\]
where $I_n$, $0_n \in M_n(\C)$ are the unit and zero matrix, respectively. 

It is well-known that the self-adjoint extensions $\mT$ of $\mS$ in $\L^2(\cG)$ can be described as follows (comp.\ \cite[Appendix~II.3]{MR1255973})). 
If we denote $\psi: \dom \mS^* \to \C^{2n}$, $\psi(\f):= \f^\sharp(v)$, then $\mS \subset \mT = \mT^* \subset \mS^*$ if and only if
$\psi(\dom \mT)$ is a maximal ${\bf J}_{2n}$-neutral subspace of~$\C^{2n}$. 
Every such subspace is of the form 
$\cL_{\cA,\cB} = \{ \bu ^\sharp = \big( \bu _1^\sharp \ \bu _2^\sharp \big)^{\rm t} \in \C^n\oplus\C^n : \cA \bu _1^\sharp + \cB \bu _2^\sharp=0\}$ 
with a pair of matrices $\cA,\,\cB \in M_n(\C)$ such that
\begin{equation}
\label{AB}
 {\rm rank}\,\big(\cA\ \ \cB\big)=n,\quad \cA \cB^*=\cB\cA^*,
\end{equation}
(see e.g.\ \cite[Lemma 2.2]{MR1671833}). This proves the following.

\begin{prop}
\label{sa}
The self-adjoint extensions of $\,\mS$ in $\L^2(\cG)$ are the restrictions $\mT_{\cA,\cB}$ of $\,\mS^*$ by 
a vertex condition 
\begin{equation}
\label{Bc}
   \cA\,\f_1^\sharp(v) + \cB \f_2^\sharp(v)=0
\end{equation}
with $\cA,\,\cB \in M_n(\C)$ satisfying \eqref{AB}. 
\end{prop}

Note that the case $\cA=I_n$, $\cB = 0_n$ corresponds to Dirichlet conditions $\f^\sharp_1(v)=0$ 
at the central vertex and hence the corresponding self-adjoint extension decouples, 
\begin{equation}
\label{T0}
\mT_0:= \mT_{I_n,0_n} = T_1 \oplus T_2 \oplus \cdots \oplus T_n. 
\end{equation}

Particular attention will be paid to the self-adjoint extensions $\mT_\tau$ given by Robin matching conditions at the central vertex~$v$,
\begin{align}
\label{k1}
f_1(v)&= f_2(v) = \cdots = f_n(v),\\
\label{k2}
f_1(v)&=
\tau \big( \wh f_1(v)+\wh f_2(v)+\cdots+\wh f_n(v) \big) \quad \text{with } \ \tau \in \R \cup \{\pm\infty\}.
\end{align} 
Here, for $\tau \ne \pm \infty$, the matrices $\cA$, $\cB$ may be chosen as
\begin{equation}
\label{Kirch}
\cA_\tau\!:=\!\begin{pmatrix}1&\!\!-1\!\!&0&\cdots&0&0\\
0&1&\!\!-1\!\!&\cdots&0&0\\
\vdots&\vdots&\vdots&&\vdots&\vdots\\
0&0&0&\cdots&1&\!\!-1\!\\
\!-1\!\!&0&0&\cdots&0&0\end{pmatrix}\!,\quad
\!\cB_\tau\!:=\tau\begin{pmatrix}0&0&\cdots&0\\
0&0&\cdots&0\\
\vdots&\vdots&&\vdots\\
0&0&\cdots&0\\
1&1&\cdots&1\end{pmatrix}\!.
\end{equation}
Note that, in accordance with the notation in \eqref{T0}, $\tau=0$ amounts to Dirichlet conditions $\f^\sharp_1(v)=0$ at $v$,
\begin{align}
\nonumber 
&f_1(v)=f_2(v)=\cdots=f_n(v)=0, \\[-1.75mm]
\intertext{while $\tau=\infty$ and $\tau=-\infty$ both amount to \vspace{-0.75mm}} 
\label{suneclipse}
&f_1(v)=f_2(v)=\cdots=f_n(v),\qquad \wh f_1(v)+\wh f_2(v)+\cdots+\wh f_n(v)=0
\end{align}
with corresponding matrices $\cA_{\pm\infty}$, $\cB_{\pm\infty}$.
We mention that, for Schr\"odinger operators, conditions of the form \eqref{suneclipse} with $\wh f_j$ 
replaced by $f_j'$ are called standard matching or Kirchhoff conditions; 
here these names may be less appropriate since the two components have equal roles and may be exchanged.

\smallskip

Later we need the following well-known lemma.

\begin{lem}\label{lemma1}
Let $\cA,\,\cB \in M_n(\C)$ be such that
\begin{equation}\label{lem1}
{\rm rank}\,(\cA\ \cB)=n,\quad \cA \cB^*-\cB\cA^*=0_n.
\end{equation}
Then there exist $\cC,\,\cD \in M_n(\C)$ such that
\begin{equation}\label{lem12}
{\rm rank}\,(\cC\ \cD)=n,\quad \cC\cD^*-\cD\cC^*=0_n,
\end{equation}
and
\begin{equation}\label{orth1}
 \cB \cC^*-\cA \cD^*=I_n.
\end{equation}
\end{lem}

\begin{proof}
If we set $\mathcal Q:=\cA\cA^*+\cB\cB^*$, it is easy to check that all conditions are satisfied e.g.\ for the matrices 
$$
\cC:=Q^{-1}\cB, \quad  \cD:=- Q^{-1}\cA.
\eqno\qedhere
$$
\end{proof}

\smallskip

Clearly, the conditions $\cA \cB^*\!-\cB\cA^*\!=\cC\cD^*-\cD\cC^*\!=0_n$ and $\cB \cC^*\!-\cA \cD^*\!=I_n$ in Lemma \ref{lemma1} mean that
\[
\begin{pmatrix}\cC&\cD \\ \cA&\cB\end{pmatrix}{\bf J}_{2n}\begin{pmatrix}\cC^*&\cA^*\\\cD^*&\cB^*\end{pmatrix}={\bf J}_{2n},
\vspace{-1mm}
\]
which \vspace{-1mm} implies
\begin{equation} 
\label{J2nunitary}
\begin{pmatrix}\cC^*&\cA^*\\\cD^*&\cB^*\end{pmatrix}{\bf J}_{2n}\begin{pmatrix}\cC&\cD \\ \cA&\cB\end{pmatrix}={\bf J}_{2n}.
\end{equation}

\vspace{2mm}

\smallskip

We remark that neither the choice of the matrices $\cA$, $\cB$ in \eqref{Bc} is unique nor is 
the choice of $\cC$, $\cD$ in Lemma \ref{lemma1} for given $\cA$, $\cB$.  
For example, it is easy to check that for the matrices $\cA_\tau$, $\cB_\tau$ in \eqref{Kirch} for $\tau \ne \pm\infty$, 
   instead of the matrices $\cC$, $\cD$ in the proof of Lemma \ref{lemma1}
   one may also choose
   \[
   \cC_\tau\!:=\!\begin{pmatrix}1&-1&0&\cdots&0&0\\0&1&-1&\cdots&0&0\\
   \vdots&\vdots&\vdots&&\vdots&\vdots\\
   0&0&0&\cdots&1&-1\\
   0&0&0&\cdots&0&0\end{pmatrix}\!, \quad
   \cD_\tau\!:=\!\begin{pmatrix}0&1&1&\cdots&1&1\\
   0&0&1&\cdots&1&1\\
   \vdots&\vdots&\vdots&&\vdots&\vdots\\0&0&0&\cdots&0&1\\
   1&1&1&\cdots&1&1\end{pmatrix}\!.
   \]


\section{Regular Dirac-Krein systems on an interval}
\label{sec:CanonicalSystems} 

In this section we collect some basic results for Dirac--Krein systems on an edge $e_j=[v,v_j]$ with $j\in\{1,2,\dots,n\}$ fixed; for simplicity, we write $x$ for the variable on $e_j$ here. 

Assume $p_j$, $q_j \in L^1(e_j,\R)$ and let $J$, $V_j$ be as in \eqref{JVj}. We consider the regular problem
\begin{equation}
\label{eq}
  -J\f_j'(x)+V_j(x)\f_j(x)-z\f_j(x)=\g_j(x),\quad \g_j\!\in\! {\bf L}^2(e_j),\quad x\!\in\! e_j,\, z\!\in\!\mathbb C,
\end{equation}
for $\f_j=\big(f_j\ \wh f_j\big)^{\rm t} \in \L^2(e_j) \oplus \L^2(e_j)$ 
subject to the boundary conditions
\begin{equation}
\label{bc}
 f_j(v)=0,\qquad 
 \cos\alpha_j\,f_j(v_j)+\sin \alpha_j \,\hf_j (v_j)=0 \ \ \text{with } \ \al_j\in[0, \pi).
\end{equation}
Using the self-adjoint operator $T_j$ introduced in \eqref{Tj}, the inhomogeneous problem \eqref{eq}, \eqref{bc} can be written as
\[
 (T_j-z)\f_j = \g_j, \quad \f_j \in \dom T_j.
\]
In the sequel we derive formulas for the resolvent of the self-adjoint operator $T_j$ in $\L^2(e_j)$ defined in \eqref{Tj},
as well as a trace formula for $T_j$.
\medskip

\subsection{}
Denote by $Y_j(\cdot,z)$, $z\in\mathbb C$, the fundamental matrix of \eqref{eq} with initial value $I_2$ at $v$, i.e.\ the unique
$2\times 2$--matrix function solving the initial value problem
\begin{equation}
\label{Y}
  -JY_j'(x,z)+V_j(x)Y_j(x,z)=zY_j(x,z),\ \ x \in e_j, \quad Y_j(v,z)=I_2.
\end{equation}
The function $Y_j(x,\cdot)$ is entire for each $x\in e_j$, satisfies $Y_j(\cdot,\ol z)=\ol{Y_j(\cdot,z)}$, and has the property
\begin{equation}
\label{basic}
  J-Y_j(x,\zeta)^*JY_j(x,z)=({z-\ol\zeta })\displaystyle\int_v^x Y_j(\xi,\zeta)^*Y_j(\xi,z)\,\d\xi,
  \quad 
  z,\zeta\in\mathbb C;
\end{equation}
this follows easily from differentiating both sides and using \eqref{Y}. 
In particular, \eqref{basic} for $\zeta=\ol z$ shows that
\begin{equation}
\label{lag}
  J=Y_j(x,\ol z)^*JY_j(x,z)=Y_j(x,z)JY_j(x,\ol z)^*,\quad x\in e_j,\ z\in\mathbb C.
\end{equation}

\medskip

We introduce the \emph{Weyl--Titchmarsh function} $m_j$ of the $j$-th edge, by the 
equation
\begin{equation}
\label{m}
  \big(\cos\al_j\ \sin\al_j\big)Y_j(v_j,z)\begin{pmatrix}-1\\m_j(z)\end{pmatrix}=0, 
  \quad z\in \C\setminus \R.
\end{equation}
Then $m_j$ is meromorphic since, with $Y_j(\cdot,z)=:\left(y_{j,kl}(\cdot,z)\right)_{k,l=1}^2$, 
\begin{equation}
\label{c}
  m_j(z)=\dfrac{y_{j,11}(v_j,z)\cos\al_j+y_{j,21}(v_j,z)\sin\al_j}{y_{j,12}(v_j,z)\cos\al_j+y_{j,22}(v_j,z)\sin\al_j} 
        =: \frac{b_j(z)}{c_j(z)},
  \quad z\in \C\setminus \R.
\end{equation}
The relation \eqref{m} can equivalently be written as
\begin{equation}
\label{m1}
  Y_j(v_j,z)\begin{pmatrix}-1\\m_j(z)\end{pmatrix}=\begin{pmatrix} -\sin\al_j\\ \hspace{3mm} \cos\al_j\end{pmatrix} \dfrac{1}{c_j(z)}, 
  \quad z\in \C\setminus \R;
\end{equation}
in fact, by \eqref{c} and \eqref{lag},
\begin{align*}
&c_j(z) Y_j(v_j,z)\begin{pmatrix}-1\\m_j(z)\end{pmatrix} \\[-3mm]
&\hspace{3mm} = Y_j(v_j,z)\begin{pmatrix}\!-c_j(z)\\ \ \,b_j(z)\end{pmatrix}
= \overbrace{Y_j(v_j,z) J Y_j(v_j,\ol z)^*}^{=J} \begin{pmatrix} \cos\al_j\\ \sin\al_j\end{pmatrix} 
= \begin{pmatrix} \!-\sin\al_j\\ \hspace{3mm} \!\cos\al_j\end{pmatrix}.
\end{align*}

\begin{prop}
\label{kreinresolventj}
For $z\in\C\setminus\R$ and arbitrary $\g_j\in {\bf L}^2(e_j)$, the solution of 
the inhomogeneous problem \eqref{eq}, \eqref{bc}, i.e.\ the resolvent $(T_j-z)^{-1}$ is given by
\begin{align}
\label{sol}
&\!\!\!\big( (T_j\!-\!z)^{-1} \g_j \big) (x)  \\[3mm]
&\!\!\!= Y_j(x,z) \!\left( \!\!\begin{pmatrix}0\!\!&\!\!-1\\0\!\!&\!\!m_{j}(z)\end{pmatrix}\!\!\displaystyle\int_v^x \!\!\!\!\!Y_{\!j\!}(\xi,\ol z)^*\!\g_{\!j\!}(\xi)\d\xi
\!+\!\begin{pmatrix}0\!\!&\!\!0\\\!-1\!\!&\!\!m_{j}(z)\end{pmatrix}\!\!\displaystyle\int_x^{v_j}\!\!\!\!\!\!Y_{\!j\!}(\xi,\ol z)^*\!\g_{\!j\!}(\xi)\d\xi\! \right)\!\!  
\nonumber
\end{align}
for $x\in e_j=[v,v_j]$. 
\end{prop}

\begin{proof}
We sketch the proof for the convenience of the reader. 
A straightforward calculation using \eqref{Y} and \eqref{lag} shows that, for $\f_j = (T_j-z)^{-1}\g_j$,
\begin{align*}
-J\f_j'(x)+&V_j(x)\f_j(x)-z\f_j(x)\\
&=-JY_j(x,z)\left( \begin{pmatrix}0\!\!&\!\!-1\\0\!\!&\!\!m_j(z)\end{pmatrix} - \begin{pmatrix}0\!\!&\!\!0\\-1\!\!&\!\!m_j(z)\end{pmatrix} \right)
Y_j(x,\ol z)^*\g_j(x)\\
&=-JY_j(x,z)JY_j(x,\ol z)^*\g_j(x)=\g_j(x).
\end{align*}
The boundary condition $f_j(v) \!=\!0$ is satisfied since the first integral in \eqref{sol} vanishes at $x\!=\!v$ 
and the first component of the second term is $0$.~The bound\-ary condition
$\cos\alpha_jf_j(v_j)+\sin \alpha_j \hf_j (v_j)\!=\!0$ is satisfied since, by \eqref{m}, it is sa\-tisfied by the first term  and the second integral in \eqref{sol} vanishes at~$x\!=\!v_j$.
\end{proof}

In the following scalar- and matrix-valued Nevanlinna functions will play an important role, so we recall the notion for convenience. 
An $M_n(\mathbb C)$-valued function $\mathcal M$ is a \emph{matrix Nevanlinna function}, or Nevanlinna function if $n=1$, if it is defined and 
locally holomorphic at least on $\C\setminus\R$ with 
\begin{equation}
\label{matrixNevanlinna}
\mathcal M(\overline z)=\mathcal M(z)^*, \quad \dfrac{\Im \mathcal M(z)}{\Im z}\ge 0,\quad z\in\C\setminus\R;
\end{equation}
the maximal domain of $\mathcal M$ consists of $\C\setminus\R$ together with all real points into which $\mathcal M$ can be continued analytically.
A matrix Nevanlinna function admits an integral representation
\[
\mathcal M(z)=\mathcal M_0 + \mathcal M_1 z+\int_{\mathbb R}\left(\frac 1{\la -z}-\frac \la{1+\la^2}\right) \d \Sigma(\la),\quad z\in\mathbb C\setminus {\rm supp}\, \Sigma,
\]
where $\Sigma$ is a non-decreasing function on $\mathbb R$ whose values are symmetric matrices in $M_n(\mathbb C)$ such that $\displaystyle\int_\R \dfrac{\d\Sigma(\la)}{1+\la^2}<\infty$,  
${\rm supp}\, \Sigma$ denotes the set of points of increase of $\Sigma$, and $\mathcal M_0, \mathcal M_1 \in M_n(\C)$ are symmetric matrices with $\mathcal M_1\ge 0$. 
The matrix Nevanlinna functions appearing in this paper are all meromorphic, i.e.\ the integral above is in fact a sum, and all their singularities are poles of first order with negative residue.

Further, we recall that a kernel $N:\cD \times \cD\to \C$ with $\cD\subset \C$ is called \emph{positive definite} 
if for all $n\in\N$, $z_i\in\cD$, $w_i\in\C$, \vspace{-1mm} $i=1,2,\dots,n$,
\[
  \sum_{i,k=1}^n \!N(z_i,z_k) \, w_i \overline{w_k} \,\ge\, 0.
\] 

In order to show that $m_j$ in \eqref{c} is a Nevanlinna function, 
for points $z\in\C$ where $m_j$ is holomorphic we define $\Gamma_{j,z}: \C \to \L^2(e_j)$ by
\[
  \Gamma_{j,z}c := Y_j(\cdot,z) \binom {-1}{m_j(z)}c, \quad \ c\in \C.      
\]

\begin{lem}
\label{nevanlinna}
The function $m_j$ has the property that 
\begin{equation}
\label{nationaldag}
  N_{m_j}(z,\zeta):=\dfrac{m_j(z)-\ol{m_j(\zeta)}}{z-\ol \zeta}  = \Gamma_{j,\zeta}^* \Gamma_{j,z}, \quad z,\zeta \in \C\setminus \R, \ z \ne \overline \zeta.
\end{equation}
Hence $N_{m_j}$ is a positive definite kernel and $m_{j}$ is a Nevanlinna~function with $\Im m_j(z)>0$ for $\Im z >0$.
 \end{lem}

\begin{proof}
Using \eqref{m1} and \eqref{basic}, we find
\begin{align*}
&\dfrac{m_j(z)-\ol{m_j(\zeta)}}{z-\ol \zeta}=\dfrac{\big(\!-\!\!1\ \ \ol{m_j(\zeta)}\big)\,J\begin{pmatrix}-1\\ m_j(z)\end{pmatrix}}{z-\ol \zeta}\\
&\!=\!\frac1{\ol{c_j(\zeta)}}\big(\!\!-\!\sin\!\al_j\, \cos\al_j\big)\dfrac{Y_j(v_j,\zeta)^{-*}JY_j(v_j,z)^{-1}}{z-\ol \zeta}\begin{pmatrix}-\sin\al_j\\ \hspace{3mm}\cos\al_j\end{pmatrix}\frac1{c_j(z)}\\
&\!=\!\frac1{\ol{c_j(\zeta)}}\big(\!\!-\!\sin\!\al_j\, \cos\al_j\big)\dfrac{Y_j(v_j,\zeta)^{-*}JY_j(v_j,z)^{-1}-J}{z-\ol \zeta}\begin{pmatrix}-\sin\al_j\\\hspace{3mm} \cos\al_j\end{pmatrix}\frac1{c_j(z)}\\
&\!=\! \big(\!-\!\!1\ \ \ol{m_j(\zeta)}\big) \dfrac{J\!-\!Y_j(v_j,\zeta)^*JY_j(v_j,z)}{z-\ol \zeta}  \begin{pmatrix}-1\\ m_j(z)\end{pmatrix}\\
&\!=\! \big(\!-\!\!1\ \ \ol{m_j(\zeta)}\big) \displaystyle\int_v^{v_j} Y_j(\xi,\zeta)^*Y_j(\xi,z)\,\d\xi \begin{pmatrix}-1\\ m_j(z)\end{pmatrix}\\
&\!=\!  \Gamma_{j,\zeta}^{*} \Gamma_{j,z}.
\end{align*}
Further, for all $n\in\N$, $z_i\in\C\setminus \R$, $w_i\in\C$, $i=1,2,\dots,n$,
\begin{align*}
\sum\limits_{i,k=1}^n \! \Gamma_{j,z_k}^* \Gamma_{j,z_i} w_i \overline{w_k} 
= \bigg( \sum\limits_{i=1}^n \Gamma_{j,z_i}  w_i , \sum\limits_{i=1}^n\Gamma_{j,z_i}  w_i \bigg)_{\L^2(e_j) } \ge 0. &
\end{align*}
Since $\Gamma_{j,z} \ne 0$, the last claim follows from \eqref{nationaldag}.
\end{proof}

\subsection{}
The resolvent $(T_j-z)^{-1}$, $z\in \rho(T_j)$, is a Hilbert--Schmidt operator since it is an integral operator with square integrable kernel on $e_j\times e_j$,
\[
\big( (T_j\!-\!z)^{-1} \g_j \big) (x)  =
\int_v^{v_j} \!\!\!\!G_j(x,\xi;z)\g_j(\xi)\,\d\xi, \quad \g_j \in \L^2(e_j), \ x\in e_j,  
\]
with Green's function
\begin{equation}
\label{green}
G_j(x,\xi;z):= 
\begin{cases} 
   Y_j(x,z) \begin{pmatrix}\ 0&\!-1\\ \ 0&\!m_j(z)\end{pmatrix} 
   Y_j(\xi,\ol z)^*,&\ \ v\le \xi<x \le v_j,\\[4mm]
   Y_j(x,z) \begin{pmatrix}0\!\!&\!\!0\\-1\!\!&\!\!m_j(z)\end{pmatrix} 
   Y_j(\xi,\ol z)^*,&\ \ v \le x < \xi \le v_j.
\end{cases}   
\end{equation}
Consequently, for $z_1, z_2\in \rho(T_j)$, the difference
\begin{equation}
\label{diff-trace}
   (T_j-z_1)^{-1}-(T_j-z_2)^{-1}=(z_1-z_2)(T_j-z_1)^{-1}(T_j-z_2)^{-1}
\end{equation}
is a trace class operator. Since, by \eqref{green} and \eqref{lag}, 
\begin{equation}
\label{resdif}
G_j(x,x+0;z)-G_j(x,x-0;z)=Y_j(x,z) (-J)
Y_j(x,\ol z)^*
=-J, \ \ x \in e_j,
\end{equation}
the resolvent difference in \eqref{diff-trace} is an integral operator with continuous kernel. 
Hence its trace can be calculated as (see e.g.\ \cite{MR0246142})
\begin{equation}
\label{resdiff}
   \tr\left((T_j\!-\!z_1)^{-\!1}\!\!-\!(T_j\!-\!z_2)^{-\!1}\right)\!=\tr\!\!\int_v^{v_{\!j}}\!\!\!\!\! \big(G_j(x,x+0;z_1)-G_j(x,x+0;z_2)\big)\d x.
\hspace{-1mm}
\end{equation}

\begin{prop}
\label{traceformula1}
For $z\in \rho(T_j)$, let $Y_j(\cdot,z)=\left(y_{j,kl}(\cdot,z)\right)_{k,l=1}^2$ and let 
\[
  c_j(z)= y_{j,12}(v_j,z)\cos\al_j+y_{j,22}(v_j,z)\sin\al_j
\]   
be the denominator in the Weyl--Titchmarsh function $m_j$ in \eqref{c}. Then
\begin{equation}
\label{trace0}
\tr \big((T_j-z_1)^{-1}-(T_j-z_2)^{-1}\big)
= - \frac{\dot c_j(z_1)}{c_j(z_1)} + \frac{\dot c_j(z_2)}{c_j(z_2)}, \quad z_1, z_2\in \rho(T_j).
\end{equation}
\end{prop}

In the proof of Proposition \ref{traceformula1} we use the following lemma.

\begin{lem} 
\label{intyy}
If we set $\dot Y_j(x,z):=\dfrac{\partial Y_j(x,z)}{\partial z}$, then
\[
   \int_v^{v_j} Y_j(x,\ol z)^*Y_j(x,z)\,\d x = -Y_j(v_j,\ol z)^* J \dot Y_j(v_j,z), \quad z\in\C.
\]
\end{lem}

\begin{proof}
Taking partial derivatives with respect to $z$ in 
\eqref{Y} for $Y_j(x,z)$, we find that, for $z\in\mathbb C$, the function $\dot Y_j(x,z)$ satisfies the initial problem
\[
  -J\dot Y_j'(x,z)+V_j(x)\dot Y_j(x,z)-z\dot Y_j(x,z)=Y_j(x,z),\ \ x\in e_j,\quad\dot Y_j(v,z)=0.
\]
A straightforward computation shows that
\[
  Y_j(x,\ol z)^*Y_j(x,z)=-\dfrac \d{\d x}\left(Y_j(x,\ol z)^*J\dot Y_j(x,z)\right),
\]
and integration from $v$ to $v_j$ yields the claim.
\end{proof}

\smallskip

\noindent
\emph{Proof of Proposition {\rm \ref{traceformula1}}.}
Let $z\in \rho(T_j)$. Because of \eqref{resdiff}, 
we first consider
\begin{align}
\nonumber
\tr\int_v^{v_j} \!\! G_j(x,x+0;z)\,\d x
&=\tr\int_v^{v_j} \!\!Y_j(x,z) \begin{pmatrix}0\!\!&\!\!0\\-1\!\!&\!\!m_j(z)\end{pmatrix} Y_j(x,\ol z)^*\,\d x\\\nonumber
&=\int_v^{v_j} \!\! \tr\,\Big(Y_j(x,z) \begin{pmatrix}0\!\!&\!\!0\\-1\!\!&\!\!m_j(z)\end{pmatrix} Y_j(x,\ol z)^*\Big) \,\d x\\\nonumber
&=\int_v^{v_j} \!\! \tr\,\bigg(Y_j(x,\ol z)^*Y_j(x,z) \begin{pmatrix}0\!\!&\!\!0\\-1\!\!&\!\!m_j(z)\end{pmatrix} \bigg) \,\d x\\\nonumber
&=\tr\int_v^{v_j} \!\! Y_j(x,\ol z)^*Y_j(x,z) \begin{pmatrix}0\!\!&\!\!0\\-1\!\!&\!\!m_j(z)\end{pmatrix} \,\d x\\\nonumber
&=\tr\left( \bigg(\int_v^{v_j} \!Y_j(x,\ol z)^*Y_j(x,z) \,\d x \bigg) \begin{pmatrix}0\!\!&\!\!0\\-1\!\!&\!\!m_j(z)\end{pmatrix} \right).
\end{align}
By Lemma \ref{intyy}, \eqref{m1}, and the definition of $c_j$ in \eqref{c}, we conclude that
\begin{align*}
\tr\!\int_v^{v_j} \!G_j(x,x+0;z)\,\d x&=-\tr\left(Y_j(v_j,\ol z)^*J\dot Y_j(v_j,z)\begin{pmatrix}0\!&\!0\\-1\!&\!m_j(z)\end{pmatrix}\right)\\
&=-\tr\left(\begin{pmatrix}0\!&\!0\\-1\!&\!m_j(z)\end{pmatrix}Y_j(v_j,\ol z)^*J\dot Y_j(v_j,z)\right)\\
&=- \dfrac 1{c_j(z)}
\tr\left(\!\begin{pmatrix}0\!&\!0\\-\sin\al_j\!\!&\!\cos\al_j\end{pmatrix}J\dot Y_j(v_j,z)\!\right)\\
&=- \frac{\dot c_j(z)}{c_j(z)}.
\end{align*}
Now the claim follows from \eqref{resdiff}. \hfill \qed

\vspace{2mm}

\begin{rem}
The trace formula \eqref{trace0} can be used to find a formula for a ``regularized trace" of $(T_j-z)^{-1}$. In fact, 
since $T_j(T_j^2+I)^{-1} = \frac{1}{2}(T_j+\i)^{-1} + \frac{1}{2}(T_j-\i)^{-1}$ and $Y_j(\cdot,-\i) = \ol{Y_j(\cdot,\i)}$,
we have
\[
\tr\left((T_j-z)^{-1}- T_j(T_j^2+I)^{-1}\right)
=  - \frac{\dot c_j(z)}{c_j(z)} + {\rm Re} \, \frac{\dot c_j(\i)}{c_j(\i)}.
\]
If we denote the eigenvalues of $T_j$, which are all simple,  by $\mu_{j,k}$, $k\in \Z$, this relation becomes
\begin{align*}
\sum_{k=-\infty}^\infty\left(\dfrac 1{\mu_{j,k}-z}-\dfrac{\mu_{j,k}}{1+\mu_{j,k}^2}\right)
= - \frac{\dot c_j(z)}{c_j(z)} + {\rm Re} \frac{\dot c_j(\i)}{c_j(\i)};
\end{align*}%
note that the right hand side is a Nevanlinna function since so is the left hand side.
\end{rem}

\section{Dirac-Krein systems on a star graph}
\label{s4}
In this section we study the Dirac--Krein systems on a star graph $\mathcal G$ introduced in Section \ref{sec:symmvertex}.
We prove a resolvent formula describing all canonical self-adjoint extensions of the symmetric operator $\mS=S_1\oplus S_2\oplus\cdots\oplus S_n$ 
with respect to the self-adjoint extension $\mT_0= \mT_{I_n,0_n} =T_1\oplus T_2\oplus\cdots\oplus T_n$ given by Dirichlet conditions $\f_1^\sharp(v)=0$ at the central vertex. 

Since the operator $\mT_0$ decouples {into independent operators on the edges, so does its resolvent $(\mT_0\!-\!z)^{-1}\!=\!(T_1\!-\!z)^{-1}\!\oplus (T_2\!-\!z)^{-1}\!\oplus\cdots\oplus (T_n\!-\!z)^{-1}$. The corresponding fundamental matrix $\Y(\cdot,z)$ with $\Y(v,z)= I_{2n}$ is block-diagonal, 
\begin{equation}
\label{bfY}
 \Y(\cdot,z) := \diag \big( Y_1(\cdot,z),  Y_2(\cdot,z), \dots, Y_n(\cdot,z) \big), \quad z \in \C. 
\end{equation}
For the extensions $\mT_{\cA,\cB}$ with boundary condition $\cA\f_1^\sharp(v) \!+\! \cB \f_2^\sharp(v)=0$ at~$v$,
it is more convenient to reorder the components of $\Y$, using the notation \eqref{perm},~\eqref{fsharp}.

In the following the matrix function $\cM$ given by
\[
 \cM(z) := \diag \big( m_1(z), m_2(z), \dots, m_n(z) \big), \quad z\in \C\setminus \R,
\]
plays a crucial role. 
Clearly, $\cM$ is a matrix Nevanlinna function (see \eqref{matrixNevanlinna}) since all $m_j$ are Nevanlinna functions by Lemma \ref{nevanlinna}.
The properties \eqref{AB} of $\cA$, $\cB$ and the fact that $\Im \cM(z)$ is strictly positive for $\Im z >0$ imply that
$\cB \cM(z) - \cA$ is invertible for all $z\in \C\setminus \R$.

\begin{prop}
\label{resolvent-Dir}
Define $ \Y(\cdot,z)^\sharp \!:=\!  \P \Y(\cdot,z)\P^{\rm t}$. Then, for $z\in\C\setminus\R$ and arbitrary $\g = (\g_j)_1^n \in {\bf L}^2(\cG)$, 
the resolvent $(\mT_0-z)^{-1}$ satisfies 
\begin{equation}
\label{sol-dir-graph}
\hspace{-0.5mm}
\begin{array}{rl}
\!\!\!
\big( (\mT_0\!-\!z)^{-1} \g \big)^{\!\sharp} (\x) 
\!\!\!\!\!\!&=\! \Y(\x,z)^\sharp 
\!\left( \! \begin{pmatrix}\ \,0_n&-I_n\\\ \,0_n&\!\!\cM(z)\end{pmatrix}\!\! {\left( \displaystyle\int_v^{x_j} \!\!\!\!\!Y_{\!j}(\xi_j,\ol z)^*\!\g_{\!j}(\xi_j)\d\xi_j \right)_{\!\!1}^{\!\!n}}^{\, \sharp} \right. \\
&\hspace{17mm} \left. + \!\begin{pmatrix}0_n\!&\!0_n\\\!-I_n\!&\!\!\!\cM(z)\end{pmatrix}\!\! {\left( \displaystyle\int_{x_j}^{v_j}\!\!\!\!\!Y_{\!j}(\xi_j,\ol z)^*\!\g_{\!j}(\xi_j)\d\xi_j\! \right)_{\!\!1}^{\!\!\!n}}^{\, \sharp} \right)\!\!  \hspace{-8mm}
\end{array}
\hspace{-1mm}
\end{equation}
for $\x\!=\!(x_j)_1^n$, $x_j\!\in\!e_j\!=\![v,v_j]$, while the resolvent $(\mT_{\cA,\cB}-z)^{-1}$ satisfies
\begin{equation}
\label{sol-genbc-graph}
\begin{array}{l}
\!\big( (\mT_{\cA,\cB}\!-\!z)^{-1} \g \big)^\sharp (\x) = \big( (\mT_0\!-\!z)^{-1} \g \big)^\sharp (\x) 
\!- \Y(\x,z)^\sharp \cdot
\\[1mm]
\cdot\!\begin{pmatrix}-I_n\\ \cM(z)\end{pmatrix} \!\big(\cB\cM(z)\!-\!\cA\big)^{-1}\cB \Big(\!\!-\!I_n \ \,\cM(z) \Big) \!
{\left( \displaystyle \int_v^{v_j}\!\!\!\!\!Y_j(\xi_j,\ol z)^* \g_j(\xi_j)\d\xi_j \right)_{\!\!1}^{\!\!n}}^{\sharp}\!\!\!. 
\end{array}
\end{equation}
\end{prop}

\begin{proof}
The formula for $(\mT_0-z)^{-1}$ follows if we take the direct sum of the formulas for $(T_j-z)^{-1}$ in \eqref{sol}, apply the permutation matrix $\P$ from the left, insert 
$\P^{\rm t} \P = \P^{-1} \P = I_n$ as a factor everywhere in between, and use that e.g.
\begin{equation}
\label{Msharp}
  \P \diag \left(  \begin{pmatrix}\!0\!&\!\!-1\\\!0\!&\!\!m_1(z)\end{pmatrix}, \dots, 
  \begin{pmatrix} \!0\!&\!\!-1\\\!0\!&\!\!m_n(z)\end{pmatrix} \right) \P^{\rm t}=   \begin{pmatrix}0_n&-I_n\\0_n&\cM(z)\end{pmatrix}.
\end{equation}

To prove the claim for $(\mT_{\cA,\cB}-z)^{-1}$, we denote the right hand side of \eqref{sol-genbc-graph} by $\f^\sharp = (\f_1^\sharp \ \f_2^\sharp)^{\rm t}$,
$\f_1^\sharp = (f_j)_1^n$, $\f_2^\sharp = (\hf_j)_1^n$ (see \eqref{perm1}).
Since $\Y(\cdot,z)$ solves the homo\-geneous differential equation on $\cG$ and $(\mT_0-z)^{-1} \g$ satisfies the inhomogeneous differential equation,  
so does $\f$.

At each vertex $v_j$ the function $\f^\sharp$ satisfies the 
boundary condition $\cos\alpha_j\,f_j(v_j)+\sin \alpha_j \,\hf_j (v_j)=0$ 
since, by \eqref{m}, it is satisfied by the first term of $\big( (\mT_0\!-\!z)^{-1} \g \big)^\sharp (\x)$ in \eqref{sol-dir-graph} as well as by the last term in \eqref{sol-genbc-graph}, and
the last integral in \eqref{sol-dir-graph} vanishes at $x_j=v_j$.

At the central vertex $v$ the boundary condition is satisfied since the first integral in \eqref{sol-dir-graph} vanishes at $x_j=v$,  
$\Y(v,z)=I_{2n}$, and hence
\begin{align*}
 &\cA \,\f_1^\sharp(v) + \cB \f_2^\sharp(v) \\[-0.05mm]
 \!&=\! \bigg( \! \Big( \!\!-\!\!\cB \ \,\cB \cM(z) \Big) \!-\! ( \cB \cM(z) \!-\! \cA ) ( \cB \cM(z) \!-\! \cA )^{-1} \cB \,\Big( \!\!-\!I_n \ \,\cM(z) \Big) \!\bigg) I(\g,z)^\sharp \!=\! 0,
\vspace{-1mm} 
\end{align*}
where we have \vspace{-1mm} set 
\begin{equation}
\label{integral}
I(\g,z):= \Big( \int_v^{v_j} \!\! Y_j(\xi_j,\ol z)^*\g_j(\xi_j)\d\xi_j \Big)_1^ n, \quad \g\in \L^2(\cG),
\end{equation}
and hence $\f\in \dom \mT_{\cA,\cB}$.
\end{proof}

The following alternative formula for the resolvent $(T_{\cA,\cB}\!-\!z)^{-1}$ gathers all singularities in one term; it is an immediate consequence of 
\eqref{sol-genbc-graph} if we multiply by $\P^{-1}=\P^{\rm t}$ from the left and use $\P^{\rm t} \Y(\cdot,\la)^\sharp = \Y(\cdot,\la) \P^{\rm t}$.

\begin{cor}
\label{PW}
For $z\in\C\setminus\R$, $\g\in {\bf L}^2(\cG)$, $\x\!=\!(x_j)_1^n$, $x_j\!\in\!e_j\!=\![v,v_j]$, we have 
\[
\begin{array}{rl}
\!\!\big( (\mT_{\cA,\cB}\!-\!z)^{-1} \g \big) (\x) 
= \!\!\!\! & \Y(\x,z) \P^{\rm t} \!\Bigg( \!\begin{pmatrix}0_n\!\!&\!\!-I_n\\0_n\!\!&\!\!0_n\end{pmatrix}\!\P\!\left(\displaystyle\int_v^{x_j} \!\!\!Y_j(\xi_j,\ol z)^*\!\g_j(\xi_j)\d\xi_j \right)_{\!\!1}^{\!\!n} \\[3mm]
& \hspace{14.8mm} +\begin{pmatrix}0_n\!\!&\!\!0_n\\\!-I_n\!\!&\!\!0_n\end{pmatrix}\!\P \bigg( \displaystyle\int_{x_j}^{v_j}\!\!\!Y_j(\xi_j,\ol z)^*\!\g_j(\xi_j)\d\xi_j \bigg)_{\!\!1}^{\!\!n} \, \Bigg) \\[3mm]
& \,+ \,\Y(\x,z) \P^{\rm t} \cP(z) \P {\left( \displaystyle\int_v^{v_j} \!\!\!\!\!Y_j(\xi_j,\ol z)^*\!\g_j(\xi_j)\d\xi_j \right)_{\!\!1}^{\!\!n}}
\end{array}
\]
where $\cP$ is the meromorphic $2n\times 2n$-matrix function given by
\begin{equation}
\label{P}
\cP(z) := \!\begin{pmatrix}0_n\!\!&\!\!0_n\\0_n\!\!&\!\!\cM(z)\end{pmatrix}
- \begin{pmatrix}-I_n\\\!\cM(z)\!\end{pmatrix}\big(\cB\cM(z)\!-\!\cA\big)^{-1} \cB \,\big(\!-\!I_n \ \cM(z)\, \big).  
\hspace{-3mm}
\end{equation}
\end{cor}
\begin{lem}
\label{lem:P}
For $z$,~$\zeta \in \C\setminus \R$, $z \ne \overline \zeta$, 
\begin{align*}
&\frac{\cP(z)-\cP(\zeta)^*}{z-\ol{\zeta}} \\
& =\! \binom{\cB^*}{-\cA^*} \big(\cM(\zeta)^*\!\cB^*\!\!\!-\!\cA^*\big)^{\!-\!1}  \dfrac{\cM(z)-\cM(\zeta)^*}{z-\ol \zeta} \big(\cB\,\cM(z)\!-\!\cA\big)^{\!-\!1} \big( \cB \  -\!\cA \big).
\end{align*}
\end{lem}

\begin{proof}
First we note that for any ${\mathcal X}\!\in\!M_n(\C)$ such that $\cB {\mathcal X} \!-\! \cA$ is invertible, by~\eqref{AB},
\begin{equation}
\label{X}
 (\cB {\mathcal X} -\cA)^{-1} \cB = \cB^* ( {\mathcal X} \cB^* - \cA^*)^{-1}.
\end{equation}
Let $z$,~$\zeta \!\in\! \C\!\setminus\!\R$, $z \!\ne\! \overline \zeta$.
The equality of the left upper entries in the claimed matrix identity follows since, by~\eqref{X},
\begin{align*}
&- (\cB\cM(z)-\cA)^{-1} \cB + \cB^* (\cM(\zeta)^*\cB^*-\cA^*)^{-1} \\
&= - \cB^* ( \cM(z) \cB^* - \cA^*)^{-1}  + \cB^* (\cM(\zeta)^*\cB^*-\cA^*)^{-1} \\
&= \cB^* (\cM(\zeta)^*\cB^*-\cA^*)^{-1} \big( \cM(z) - \cM(\zeta)^* \big) \cB^* ( \cM(z) \cB^* - \cA^*)^{-1}\\
&= \cB^* (\cM(\zeta)^*\cB^*-\cA^*)^{-1} \big( \cM(z) - \cM(\zeta)^* \big) (\cB \cM(z) - \cA)^{-1} \cB.
\end{align*}
The equality of the right lower entries follows since, again by~\eqref{X},
\begin{align*}
&\cM(z) \!-\! \cM(z) (\cB\cM(z)\!-\!\cA)^{-\!1} \cB \cM(z)  \!\!\\
& \hspace{3.5mm}- \cM(\zeta)^* +\!  \cM(\zeta)^* \cB^* (\cM(\zeta)^*\cB^*\!\!-\!\cA^*)^{-\!1} \cM(\zeta)^*\\
&=- \cM(z) (\cB\cM(z)-\cA)^{-1} \cA + \cM(\zeta)^* (\cB \cM(\zeta)^* - \cA)^{-1} \cA\\
&= \big( - \cM(z) \!+\! \cM(\zeta)^* (\cB \cM(\zeta)^* \!-\! \cA)^{-1} (\cB\cM(z)\!-\!\cA) \big) (\cB\cM(z)\!-\!\cA)^{-1} \cA \\
&= \big( - I + \cM(\zeta)^* (\cB \cM(\zeta)^*- \cA)^{-1} \cB \big) (\cM(z) \!-\! \cM(\zeta)^* )  (\cB\cM(z)\!-\!\cA)^{-1} \cA\\
&= \big( - I + \cM(\zeta)^* \cB^* (\cM(\zeta)^*\cB^*\!-\!\cA^*)^{-1} \big) (\cM(z) \!-\! \cM(\zeta)^* )  (\cB\cM(z)\!-\!\cA)^{-1} \cA\\
& = \cA^*  (\cM(\zeta)^*\cB^*\!-\!\cA^*)^{-1} (\cM(z) \!-\! \cM(\zeta)^* )  (\cB\cM(z)\!-\!\cA)^{-1} \cA.
\end{align*}
The proof of the equality of the off-diagonal entries is similar.
\end{proof}

The following is our first main result.

\begin{thm}
\label{thm:kres}
{\rm (i)}\! The self-adjoint extensions of the symmetric operator $\mS$~in $\L^2(\mathcal G)$
are the restrictions $\mT_{\cA,\cB}$ of the operator $\mS^*$ given by vertex conditions
\[
\cA\,\f_1^\sharp(v) + \cB \f_2^\sharp(v)=0,
\]
with $\cA$, $\cB \in M_n(\C)$  such that 
\[
{\rm rank}\,\big(\cA\ \ \cB\big)=n,\quad \cA \cB^*=\cB\cA^*.
\]
{\rm (ii)} 
The resolvent of $\,\mT_{\cA,\cB}$ satisfies
\begin{equation}
\label{kres}
  \ \ 
  (\mT_{\cA,\cB}-z)^{-1}=(\mT_0-z)^{-1}-\Gamma_z(\cB\cM(z)-\cA)^{-1}\cB\,\Gamma_{\ol z}^*
\end{equation}
for $z\in\rho(\mT_0)\cap\rho(\mT_{\cA,\cB})$ with $\mT_0\!=\!\mT_{I_n,0_n}$ and $\Gamma_z\!:\!\C^n \!\to\! \L^2(\cG)$ 
\vspace{-1mm} given~by
\begin{equation}
\label{stockholm2} 
  \Gamma_z c:= \Y(\cdot,z) \,
  \P^{\rm t} \!\begin{pmatrix}-I_n\\ \cM(z)\end{pmatrix} c, \quad c \in \C^n,
\end{equation}%
where $\P$ is the permutation matrix in \eqref{perm}.
\\[1mm]
{\rm(iii)}
The fractional linear transformation $\cW_{\cA,\cB} \in M_n(\C)$ of $\cM$,
\[
   \cW_{\cA,\cB}(z)=\big(\cD\cM(z)-\cC\big)\big(\cB\cM(z)-\cA\big)^{-1}, \quad z\in\C\setminus\R,
\]
with $\cC$, $\cD \in M_n(\C)$ related to $\cA$, $\cB$ as in Lemma $\ref{lemma1}$, is a meromorphic matrix Nevanlinna function and
\begin{equation}
\label{Qfunction}
 (\mT_{\cA,\cB}-z)^{-1}= \Y(\cdot,z) \,\P^{\rm t} \binom{\cB^*}{-\cA^{*}}  \cW_{\cA,\cB}(z) \Big( \cB \  -\!\!\cA \Big) \P I(\cdot,z) + {\mathcal F}(z)
\end{equation}
where $I(\cdot,z)$ is given by \eqref{integral} and  ${\mathcal F}$ is an entire function whose values are bounded linear operators in $\L^2(\cG)$.
\end{thm}

\begin{proof}
Claim (i) was proved in Proposition \ref{sa}.
To prove (ii) we multiply \eqref{sol-genbc-graph} with $\P^{-1}=\P^{\rm t}$ from the left, note $\P^{\rm t} \Y(\cdot,\la)^\sharp = \Y(\cdot,\la) \P^{\rm t}$, and
that, for $z\in \C\setminus \R$, the adjoint $\Gamma_z^*:\L^2(\cG)\to \mathbb \C^n$ of  $\Gamma_z$  in \eqref{stockholm2} is given by
\begin{equation}
 \label{stockholm5}
 \Gamma_z^* \g =  \Big(\!\!-\!I_n \ \,\cM(\ol z) \Big) \P I(\g,\ol z), \quad \g \in \L^2(\cG).
\end{equation}
\noindent
(iii) The choice of $\cC$, $\cD$ in Lemma \ref{lemma1} ensures that  \eqref{J2nunitary} holds. 
This implies that, for $z$,~$\zeta \in \C\setminus \R$, $z \ne \overline \zeta$,
\begin{align*}
&\frac{\cW_{\cA,\cB}(z)-\cW_{\cA,\cB}(\zeta)^*}{z-\ol{\zeta}}  = \big(\cM(\zeta)^*\!\cB^*\!\!\!-\!\cA^*\big)^{\!-\!1}  \dfrac{\cM(z)-\cM(\zeta)^*}{z-\ol \zeta} \big(\cB\cM(z)\!-\!\cA\big)^{\!-\!1} 
\end{align*}
and hence $\cW_{\cA,\cB}$ is a matrix Nevanlinna function. Moreover, by Lemma \ref{lem:P},
\[
 \frac{\cP(z)-\cP(\zeta)^*}{z-\ol{\zeta}}  = \binom{\cB^*}{- \cA^*} \frac{\cW_{\cA,\cB}(z)-\cW_{\cA,\cB}(\zeta)^*}{z-\ol{\zeta}} \big( \cB \  -\!\cA \big)
\]
and thus  $\cP (z) \!=\! \displaystyle \begin{pmatrix} \! \cB^*\! \\[-1mm] \!-\cA^*\! \end{pmatrix} \! \cW_{\cA,\cB} (z)  \big( \cB \  -\!\cA \big) + X$ with  $X\!=\!X^*\!\in\! M_n(C)$ for $z\!\in\!\C\setminus\R$.  Now \eqref{Qfunction} follows 
from the represent\-ation of $(T_{\cA,\cB}\!-\!z)^{-1}$ in Corollary \ref{PW} if we note that $z\mapsto \Y(\cdot,z)$ is entire.
\end{proof}

\begin{rem}
By \eqref{stockholm5} and \eqref{bfY} we have, for $z\!\in\!\C\setminus\R$ and  $c\!=\!(c_j)_1^n \!\in\! \C^n$, \vspace{-1mm}$\g =(\g_j)_1^n \in \L^ 2(\cG)$,
\begin{align}
\label{gammafield}
  &\Gamma_z (c_j)_1^n = \Big( c_j \y_j(\cdot,z) \Big)_1^n, \ \  \y_j(\cdot,z) := Y_j(\cdot,z) \binom{-1}{m_j(z)} \in \ker (S_j^*-z), \\[-2mm]
\label{gammafield*}  
  &\Gamma_z^* \g = \Big( \int_v^{v_j} \!\!\!\y_j(\xi_j,\ol z)^{\rm t} \g_j(\xi_j) \d \xi_j \Big)_1^n. \\[-6mm] \nonumber
\end{align}
Hence $\ran \Gamma_z = \ker (\mS^*-z) = \ran (\mS - \ol z)^\perp$ is the defect space of $\mS$ at $\ol z$. 
\end{rem}

Note that, due to the non-uniqueness of the boundary matrices $\cA$, $\cB$, two parameter pairs $\cA$, $\cB$ and $\cA'$, $\cB'$ generate the same self-adjoint extension if the linear relations $\cB^{-1}\cA$ and $(\cB')^{-1}\cA'$ coincide; in this case the subspaces defined by the corresponding boundary conditions coincide.

\smallskip

The relation \eqref{kres} is a version of M.G.~Krein's resolvent formula (comp.\ \cite{MR0282239}). It describes  the resolvents of all the self-adjoint extensions of $\mS$ within $\L^2(\cG)$,  
with the decoupled extension $\mT_0$ fixed, by means of pairs of $n\times n$ matrices $\cA$, $\cB$ satisfying \eqref{AB}.
The same formula with $z$-dependent matrices $\cA$, $\cB$ describes all self-adjoint extensions of $\mS$ with exit, i.e.\ in Hilbert spaces of which $\L^2(\cG)$ is a subspace.

\begin{rem}
\label{fourier}
The resolvent formula in \eqref{Qfunction} may be used to derive eigenfunction expansions for the Dirac-Krein operator $\mT_{\cA,\cB}$. 
The poles of the meromor\-phic matrix Ne\-vanlinna function $\cW_{\cA,\cB}$ are the eigenvalues of $\mT_{\cA,\cB}$. Since the resolvent difference of $\mT_0$ and $\mT_{\cA,\cB}$ has finite rank by \eqref{kres}
and the eigenvalues of each $T_j$, whence those of $\mT_0 \!=\! \bigoplus_{j=1}^n \!T_j$, accumulate only at $-\infty$ and $+\infty$, so do the eigenvalues of~$\mT_{\cA,\cB}$;
we denote them by $(\la_k)_{k\in\Z}$ with $\la_k \!< \!\la_{k+1}$, $k\!\in\!\Z$.  The principal part of $\cW_{\cA,\cB}(z)$ at $\la_k$ is of the form $\dfrac{-W_{\cA,\cB,k}}{z\!-\!\la_k}$ with non-negative 
$W_{\cA,\cB,k} \!\in\! M_n(\C)\backslash \{0_n\}$~of rank equal to the multiplicity of $\la_k$. Then the functional calculus and the residue theorem applied to \eqref{Qfunction} imply that, for all bounded intervals $\Delta \!\subset\!\R$,
$\f\!=\!(\f_j)_1^n\!\in\! \L^2(\cG)$, and $\x\in\cG$, 
\begin{align*}
&\big(E(\Delta)\f\big)(\x) \\
&\!=\!\!\!\sum_{\la_k\in \Delta} \!\!\!\Y(\x,\la_k) \P^{\rm t} \!\!\begin{pmatrix} \cB^*\\-\cA^*\end{pmatrix}  W_{\cA,\cB,k} 
\big(\cB\ \, -\!\cA\big) \P \!\Big( \!\int_v^{v_j}\!\!\!\!Y_j(\xi_j,\la_k)^*\f_{\!j}(\xi_j)\d\xi_j \Big)_1^{n\,\sharp}\!\!. 
\end{align*}
Letting $\Delta\!=\!(-R,R)$ and $R\!\to\!\infty$, we arrive at the eigenfunction \vspace{-1mm}  expansion
\[
\f\!=\!\sum_{k=-\infty}^\infty \!\! \Y(\cdot,\la_k)  \P^{\rm t} \!\! \begin{pmatrix} \cB^*\\-\cA^*\end{pmatrix}  W_{\cA,\cB,k} 
\big(\cB\ \, -\!\cA\big) \P \! \Big( \int_v^{v_j}\!\!\!Y_j(\xi_j,\la_k)^*\f_j(\xi_j)\d\xi_j \Big)_1^{n\,\sharp} 
\]
in $\L^2(\cG)$. Hence the \vspace{-1mm}  mapping 
\[
\f \mapsto \big(\cB\ \, -\!\cA\big) \P \Big( \int_v^{v_j}\!\!\!Y_j(\xi_j,\la_k)^*\f_j(\xi_j)\d\xi_j \Big)_1^{n\,\sharp} 
= \big(\cB\  -\!\cA\big) \P I(\f,\la_k)^\sharp 
\]
can be considered as the Fourier transformation associated with $\mT_{\cA,\cB}$, and Parseval's relation takes the form
\[
(\f,\g)_{\L^2(\cG)}= \sum_{k=-\infty}^\infty I(\g,\la_k)^{\rm t}\P^{\rm t}\binom{\cB^*}{-\cA^*} W_{\cA,\cB,k} \big(\cB\  -\!\cA\big)  \P I(\f,\la_k)^{\sharp}.
\]
\ct{Unlike the single interval case $n=1$ (see e.g.\ \cite{MR3147401}), for star graphs with $n>1$ edges the Fourier transformation is no longer scalar but $n$-dimensional  
and hence higher spectral multiplicities may occur, see Theorem \ref{thmTK} in the next section.}
\end{rem}

\section{Eigenvalues and multiplicities for Robin matching conditions}
\label{sec:tau}

In this subsection, for the self-adjoint extension induced by the Robin matching conditions \eqref{k1}, \eqref{k2}  at the central vertex, we derive 
a characteristic equation for the eigenvalues in terms of the Weyl--Titchmarsh functions $m_j$ on the edges, investigate their multiplicities, 
and we prove a trace formula for the resolvent difference of $\mT_\tau$ and $\mT_0$.

To this end, we first specialize the resolvent formulas established in Theorem~\ref{thm:kres}.  
For the Robin matching conditions \eqref{k1}, \eqref{k2} with $\tau \ne \pm\infty$, we choose the matrices $\cA=\cA_\tau$, $\cB=\cB_\tau$ as in \eqref{Kirch}, 
and correspondingly $\cA_{\pm\infty}$ and $\cB_{\pm\infty}$ for $\tau=\pm\infty$.
The associated self-adjoint operator is denoted by  $\mT_\tau:= \mT_{\cA_\tau,\cB_\tau}$. 
Note that for $\tau=0$ the operator coincides with $\mT_0=\mT_{I_n,0_n}$.
Hence we assume that $\tau\ne 0$ in the following if not explicitly stated otherwise.

Straightforward calculations yield that 
\[
(\cB_\tau\cM(z)\!-\!\cA_\tau)^{\!-\!1}\cB_\tau
=\dfrac 1{\dfrac 1\tau+\!\displaystyle\sum_{j=1}^n\!m_j(z)}\!\!\begin{pmatrix}1\\1\\[-1.5mm]\vdots\\1\end{pmatrix}\!\!
\begin{pmatrix}1\!\!&\!\!1\!\!&\!\!\cdots\!\!&\!\!1\end{pmatrix}\!.
\]
Since all $m_j$, $j\!=\!1,2,\dots,n$, are Nevanlinna functions, so is the \vspace{-2mm} function 
\begin{equation}
\label{tw}
   \m_\tau(z):= \dfrac 1\tau + \sum_{j=1}^n m_j(z), \quad z\in\C\setminus\R.
\end{equation}

Using the above formula for $(\cB_\tau\cM(z)\!-\!\cA_\tau)^{\!-\!1}\cB_\tau$ and the formulas \eqref{gammafield}, \eqref{gammafield*} for $\Gamma_z$,~$\Gamma_{\ol z}^*$, 
it is immediate from Theorem \ref{thm:kres} that, for the Robin matching conditions \eqref{k1}, \eqref{k2}, Krein's resolvent formula \eqref{kres} takes the following form.

\begin{cor}
For $z\!\in\! \C\!\setminus\! \R$, $\g \!=\! (\g_j)_1^n \!\in\! {\bf L}^2(\cG)$ and $\x\!=\!(x_j)_1^n$, $x_j\!\in\!e_j\!=\![v,v_j]$, 
\begin{align}
\nonumber
  &\left((\mT_\tau-z)^{-1}\g\right)\!(\x)\\
\label{kres-standard}
  &=\left((\mT_0-z)^{-1}\g\right)\!(\x)
    - \frac 1{\m_\tau(z)}\! 
  \bigg(\displaystyle\!\sum_{\ell=1}^n\int_v^{v_\ell}\!\!\!\!\!\y_\ell(\xi_j,\ol z)^{\rm t}\g_\ell(\xi_j)\,\d\xi_j\!\bigg) \! \begin{pmatrix} \y_{\!j}(x_j,z) \end{pmatrix}_{1}^{n} \hspace{-0.5mm} \\[-1mm]
\nonumber
& \,= \left((\mT_0-z)^{-1}\g\right)\!(\x) - \frac 1{\m_\tau(z)} \big( \y(\cdot,\ol z), \g \big)_{\L^2(\cG)} \y (\x,z),
\end{align}
where $\y(\cdot,z) = (\y_j(\cdot,z))_1^n$, $\,\y_j(\cdot,z) = Y_j(\cdot,z) \displaystyle \binom {-1}{m_j(z)}$, and
\begin{equation}
\label{expl}
\begin{array}{rl}
\left((\mT_0-\ z)^{-1}\g\right)\!(\x)
= \Big( \big( (T_j-z)^{-1} \g_j \big) (x_j) \Big)_1^n
\end{array}
\end{equation}
with $(T_j-z)^{-1}\g_j$ given by \eqref{sol}.
\end{cor}

In Theorem \ref{thmTK} below we show that the set of eigenvalues splits into two parts, 
simple eigenvalues that depend on $\tau$ and possibly non-simple eigenvalues that are 
independent of $\tau$.

\begin{lem}
   \label{lem:TKeigenvaluessimple}
   Let $z\in\mathbb C$ be  not a pole of $\m_\tau$, i.e.\ not a pole of any $m_j$, $j=1,2$, $\dots,n$.
   Then 
   \begin{align}
   \label{eigenvalue1}
       z\in\sigma_{\rm p}(\mT_\tau) \ 
       & \iff \ \cB_\tau \cM(z) \!-\! \cA_\tau \mbox{ is not invertible} 
       \ \iff \ \m_\tau(z)=0. \hspace{-1mm} 
   \end{align}
   These eigenvalues of $\mT_\tau$ are simple, with eigenfunction
   $\y(\cdot,z) \!=\! \big(\y_j(\cdot,z)\big)_1^n$ given by
   $\y_{\!j}(\cdot,z)  \!=\!  u_j
   Y_{\!j}(\cdot, z)\!
   \begin{pmatrix} \!-1\! \\ \!m_j(z)\! \end{pmatrix}$
   where $\bu\!=\!(u_j)_1^n \!\ne\! 0$  
   is a solution of $\big(\cB_\tau \cM(z) \!-\! \cA_\tau \big)\bu\!=\!0$.
\end{lem}

\begin{proof}
   Let $\cA_\tau$, $\cB_\tau$ be the boundary matrices for $\mT_\tau$ defined in \eqref{Kirch}
   and assume $z\in\mathbb C$  is not a pole of any $m_j$, $j=1,2,\dots,n$; then $z$ is 
   not a pole of $\cM= {\rm diag}\,(m_1,m_2,\dots,m_n)$. 
   
   By Proposition \ref{sa}, $z \in \sigma_{\rm p}(\mT_\tau)$ with eigenfunction
   $\y(\cdot,z) = \big(\y_j(\cdot,z)\big)_1^n$ if and only if the following hold:
   
   \vspace{2mm} 
   
   \!\!\!\!(i) \,$\y(\cdot,z)\!\in\!\ker(S^*\!\!-\!z)$, \vspace{-1mm}i.e.\ 
   $  \y_j (\cdot,z)= \gamma_j \,Y_j(\cdot, z)\!
      \begin{pmatrix} -1 \\ m_j(z)\! 
      \end{pmatrix}$ with $\gamma_j\!\in\!\C$,
   \hspace*{12mm} \vspace{4mm} $j\!=\!1,2,...,n$;
   
   \!\!\!\!(ii) $\y(\cdot,z)$ satisfies the boundary condition $\cA_\tau \y_1^\sharp(v,z) + \cB_\tau \y_2^\sharp(v,z) = 0$.
    
   \vspace{3mm} 
    
   \!\!``$\Longrightarrow$'' in the first equivalence in \eqref{eigenvalue1}: 
   Since $\Y_j(v,z)=I_{2}$ and due to \eqref{Msharp}, condition (i) implies that
   $\y_2^\sharp(v,z) = - \cM(z) \y_1^\sharp(v,z)$;
   in particular, $\y_1^\sharp(v,z)\ne 0$ because $z$ is no pole of $\cM$ and $\y(\cdot,z)\ne 0$. Hence, by the condition in (ii),
   \begin{align}
      \label{eq:chareq}
      (\cB_\tau \cM(z)-\cA_\tau)\y_1^\sharp(v,z) = 0
   \end{align}
   has a non-zero solution. 
   
   \vspace{2mm}
   
   \!\!``$\Longleftarrow$'' in the first equivalence in \eqref{eigenvalue1}: 
   If \eqref{eq:chareq} has a non-zero solution $\y_1^\sharp(v,z)\!\in\! \C^n$ with components $\y_1^\sharp(v,z)_j$, $j\!=\!1,2,\dots,n$, then the function 
   $\y(\cdot,z)= (\y_j(\cdot,z))_1^n$ with components
   $\y_j(\cdot,z) = \y_1^\sharp(v,z)_j Y_j(\cdot, z)
   \!\begin{pmatrix}-1 \\ m_j(z) \end{pmatrix}$, $j=1,2,\dots,n$,
   satisfies~(i) and (ii) by construction.
   
   \vspace{1mm}
   
   The second equivalence in \eqref{eigenvalue1}, e.g.\ for $\tau \ne \infty$,
   is immediate from 
   \begin{align*}
   \cB_\tau \cM(z) - \cA_\tau 
   \!=\!
   \begin{pmatrix}
	 \hspace{8mm} -1 & 1 &  &&& 0\\
	 & -1 & 1 &  && \\
	 && -1 & \ \ 1 & & \\
	 &&& \ \ \ddots & \ddots &\\
	 \hspace{8mm} 0 &&&& \hspace{-4mm}-1 & 1\\
	 \!1+\tau m_1(z)\!\!& \tau m_2(z) & \cdots\cdots \!\!\!&\!\! \cdots & \!\!\!\tau m_{n-1}(z) & \!\!\tau m_n(z)
   \end{pmatrix}
   \end{align*}
and 
\[
  \det( \cB_\tau \cM(z) - \cA_\tau ) = (-1)^{n+1} \Big( 1 + \tau\sum_{j=1}^n m_j(z)\Big) =  (-1)^{n+1} \tau \m_\tau(z).
\]
Since the dimension of the kernel of the matrix $\cB_\tau \cM(z) - \cA_\tau$ is at most $1$, every zero of $\m_\tau$ is a simple eigenvalue of $\mT_\tau$.
\end{proof}

Apart from the simple eigenvalues characterized in Lemma \ref{lem:TKeigenvaluessimple},
there may be other eigenvalues of possibly higher multiplicity which arise as poles of 
at least two of the Weyl--Titchmarsh functions $m_j$ on the edges $e_j$.

\begin{thm}
\label{thmTK}
For the Robin matching conditions \eqref{k1}, \eqref{k2} with $\tau\ne 0$, the spectrum $\sigma_{\rm p}(\mT_\tau)$ of $\,\mT_\tau$ 
consists of two types of eigenvalues:
\begin{itemize}
\item[{\rm i)}] The solutions of the \vspace{-2mm}equation 
\begin{equation}
\label{Weyl}
   \sum_{j=1}^n m_j(z)= \ct{-} \frac 1\tau
\end{equation}
are simple eigenvalues and each of them is strictly increasing in $\tau$ on the intervals $[-\infty,0)$ and \vspace{1mm}$(0,\infty]$;
\item[{\rm ii)}] 
The common poles of $\nu+1 \in \{2,3,\dots,n\}$ functions $m_j$ are eigenvalues of multiplicity $\nu \in \{1,2,\dots,n-1\}$ and are independent of $\,\tau$. 
\end{itemize}
Further, for any $\tau_1$, $\tau_2 \in \R \cup \{\infty\}$, the eigenvalues of $\,\mT_{\tau_1}$ and of $\,\mT_{\tau_2}$~interlace.
\end{thm}

\begin{proof}
First assume that $\la_0\in\R$ is not a pole of any $m_j$, $j=1,2,\dots,n$. Then, by Lemma \ref{lem:TKeigenvaluessimple},
$\la_0 \in \sigma_{\rm p}(\mT)$ if and only if $\la_0$ is a solution of \eqref{Weyl}, and in this case $\la_0$ is simple.

Now suppose that $\la_0\!\in\!\R$ is a pole of $\nu\!+\!1$ functions $m_j$ with $\nu \!\in\! \{0,1,\dots,n\!-\!1\}$, without loss of generality, of 
$m_1$, $m_2$, $\dots$, $m_{\nu+1}$. 
The residue $\gamma_j$ of $m_j$ at $\la_0$ satisfies $\gamma_j<0$, $j\!=\!1,2,\dots,\nu+1$, 
and we set $\gamma_j\!:=\!0$, $j\!=\!\nu+2,\nu+3,\dots,n$.

The orthogonal projection onto the possible eigenspace of $\mT_\tau$ at $\la_0$ is given~by $\lim\limits_{z\to\la_0} (\la_0-z)(\mT_\tau-z)^{-1}$.
Using \eqref{kres-standard} and \eqref{sol-dir-graph}, we find that, 
for $\g = (\g_j)_1^n \in \L^2(\cG)$ and $\x\!=\!(x_j)_1^n$, $x_j\!\in\!e_j\!=\![v,v_j]$,
\begin{align}
&\lim\limits_{z\to\la_0}\!\!\left((\la_0-z)(\mT_\tau-z)^{-1}\g\right)(\x) \nonumber
\\[1mm]	
\label{res1}
&=\left(\!\!\ Y_j(x_j,\la_0)\!\begin{pmatrix}0\!&\!0\\0\!&\!\gamma_j\end{pmatrix}\!\displaystyle\int_v^{v_j}\!\!\!Y_j(\xi_j;\la_0)^*\g_j(\xi_j)\,\d\xi_j\right)_{\!1}^{\!n} 
\hspace{-6mm}\\
\nonumber
&\hspace{2mm} - \lim\limits_{z\to\la_0}\dfrac {\tau(\la_0-z)^2}{\Big(1\!+\!\tau\!\displaystyle\sum_{j=1}^nm_j(z)\Big)(\la_0\!-\!z)}\!
\begin{pmatrix}\!\!\bigg(\displaystyle\!\sum_{\ell=1}^n\int_v^{v_\ell}\!\!\!\!
\y_\ell(\xi_j,\ol z)^{\rm t}\g_\ell(\xi_j)\,\d\xi_j\!\bigg)\y_j(x_j,z)\!\end{pmatrix}_{\!\!1}^{\!\!n}\!\!\!. \\[-8mm] \nonumber
\end{align}
Since $\y_{\!j}(x_j,z)\!=\!Y_j(x_j,z)\!\begin{pmatrix}-1\\\!m_j(z)\!\end{pmatrix}$, the limit
$
\lim\limits_{z\to\la_0}\!(\la_0-z)\y_{\!j}(x_j,z)$ equals $Y_j(x_j,\la_0)\!\begin{pmatrix}0\\ \gamma_j\end{pmatrix}\! $ for 
$j=1,2,\dots,\nu\!+\!1$
and $\begin{pmatrix}0\\ 0\end{pmatrix}$ for $j=\nu+2,\nu+3,\dots,n$. So the limit on the right hand side of \eqref{res1}~equals
\begin{align*}
-\dfrac 1{\displaystyle\sum_{j=1}^{\nu+1} \gamma_j}\begin{pmatrix}\!\!\bigg(\displaystyle\sum_{\ell=1}^{\nu+1}\int_v^{v_\ell} 
    \!\big(0\ \gamma_\ell\big) \,Y_\ell(\xi_j,\la_0)^*\,\g_\ell(\xi_j)\,\d\xi_j\!\bigg)Y_j(x_j,\la_0)                                                               
    \begin{pmatrix}0\\\gamma_j\end{pmatrix}\end{pmatrix}_{\!\!1}^{\!\!n}\\[-6mm]
    =:-\dfrac{1}{\displaystyle\sum_{j=1}^{\nu+1} \gamma_{j}} 
   \begin{pmatrix} \gamma_j \bigg(\displaystyle\!\sum_{\ell=1}^{\nu+1} \gamma_\ell\mu_\ell\bigg)\begin{pmatrix}y_{j,12}(x_j,\la_0)\\y_{j,22}(x_j,\la_0)\end{pmatrix}\end{pmatrix}_{\!\!1}^{\!\!n}
\\[-7mm]
\end{align*}
with
\[
\mu_\ell:=
\int_v^{v_\ell}\!\!\! \!\big(y_{\ell,12}(\xi_j,\la_0)\ y_{\ell,22}(\xi_j,\la_0) \big)\begin{pmatrix}g_\ell(\xi_j)\\ \wh g_\ell(\xi_j)\end{pmatrix}\!\d\xi_j,
\quad \ell=1,2,\dots,{\nu+1}.
\]
The components of the first term on the right hand side of \eqref{res1} can  be written as
\[
\gamma_j\!\!\displaystyle\int_v^{v_j}\!\!\!\!\!\!\!\big(y_{j,12}(\xi_j,\!\la_0)\  y_{j,22}(\xi_j,\!\la_0)\big)\!
\begin{pmatrix}g_j(\xi_j)\\ \wh g_j(\xi_j)\end{pmatrix}\!\d\xi\! \begin{pmatrix}y_{j,12}(x_j,\!\la_0)\\y_{j,22}(x_j,\!\la_0)\end{pmatrix}
\!=\!
\gamma_j\mu_j\!\!\begin{pmatrix}y_{j,12}(x_j,\!\la_0)\\y_{j,22}(x_j,\!\la_0)\end{pmatrix}\!.
\]
Altogether, \eqref{res1} becomes
\[
\lim_{z\to\la_0}(\la_0\!-\!z)\!\left((\mT_\tau\!-\!z)^{-1}\g\right)(\x)
=\Bigg(\bigg[\gamma_j\mu_j-\dfrac{\gamma_j}{\displaystyle\sum_{\ell=1}^{\nu+1} \gamma_\ell}\sum_{\ell=1}^{\nu+1} \gamma_\ell\mu_\ell\bigg]\!\!
\begin{pmatrix}y_{j,12}(x_j,\la_0)\\y_{j,22}(x_j,\la_0)\end{pmatrix}\!\!\Bigg)_{\!\!1}^{\!\!n}. 
\]
This form shows that the dimension of the eigenspace of $\mT_\tau$ at $\la_0$ is $\le\nu$ since $\gamma_j=0$ for $j=\nu+2,\nu+3,\dots,n,$ and
\[
\sum_{j=1}^{\nu+1}\Bigg(\gamma_j\mu_j-\dfrac{\gamma_j}{\displaystyle\sum_{\ell=1}^{\nu+1} \gamma_{\ell}}\sum_{\ell=1}^{\nu+1} \gamma_\ell\mu_\ell\Bigg)=0;\vspace{-2mm}
\]
\noindent
in particular, if $\nu\!=\!0$, i.e.\ $\la_0$ is a pole of only one $m_j$, then $\la_0$ is no eigenvalue of~$\mT_\tau$.

That the dimension of $\ker (\mT_\tau-\la_0)$ cannot be $<\nu$ follows from the fact that the rank of $(\mT_\tau-z)^{-1}-(\mT_0-z)^{-1}$ is equal to $1$ by 
\eqref{kres-standard} and the multiplicity of $\la_0$ as an eigenvalue of $\mT_0$ equals $\nu+1$ by assumption.

For $\tau_1$, $\tau_2 \ne 0$, the claimed interlacing property is immediate from the fact that the left hand side of
\eqref{Weyl} is a Nevanlinna function. If e.g.\ $\tau_2=0$, 
then the eigenvalues of $\mT_{\tau_2} = \mT_0 =T_1\oplus T_2\oplus\cdots\oplus T_n$ 
are the poles of the left hand side of \eqref{Weyl} and a common pole of $k$ 
functions $m_j$ is an eigenvalue of $\mT_0$ of multiplicity $k$, and the claim follows from i) and ii).
\end{proof}

\begin{cor}
\label{cor:interlace}
For $\tau \!\ne\! 0$ two consecutive eigenvalues of $\,\mT_\tau$ cannot be multiple.
\end{cor}

\begin{rem}
For the special case $\tau=\infty$ of standard matching conditions, 
Theorem~\ref{thmTK} was proved in \cite[Thm.~3.19~(I/II)]{MR3180877} in the more general setting of symmetric relations with not necessarily discrete spectrum.
\ct{An~analogous result for star graphs of Stieltjes strings was proved in \vspace{-1mm} \cite[Cor.~2.6]{MR3091304}.}
\end{rem}

\begin{thm}
\label{thm:trace}
With $\m_\tau(z)=\dfrac 1\tau+\displaystyle\sum_{j=1}^n m_j(z)$ as in \eqref{tw}, the trace \vspace{-2mm} formula
\[
  \tr \left((\mT_\tau-z)^{-1}-(\mT_0-z)^{-1}\right) = - \dfrac{\dot \m_\tau(z)}{\m_\tau(z)}, 
  \quad z\in\C\setminus\R,
\]
holds, i.e.\ $\m_\tau$ is the perturbation determinant of $\,\mT_\tau$ with respect to $\mT_0$.
\end{thm}

\begin{proof}
Let $z\in\C\setminus \R$. 
By \eqref{kres-standard} the resolvent difference $(\mT_\tau\!-\!z)^{-1}-(\mT_0\!-\!z)^{-1}$ is an integral operator in $\L^2(\cG)$ of rank 1 with semi-separable kernel 
\[
  K_\tau(\x,\bxi;z) :=  - \frac 1{\m_\tau(z)} \y(\x,z) \y(\bxi,\ol z)^{\rm t}, \quad \x,\bxi \in \cG,
\]
hence its trace is given by \vspace{-1mm} (comp.\ e.g.\ \cite[Thm.\ IX.3.2]{MR1130394})
\begin{equation}
\label{mjdot}
  \tr \left((\mT_\tau\!-\!z)^{-1}-(\mT_0\!-\!z)^{-1}\right)= 
  - \frac 1{\m_\tau(z)} \,
  \sum_{j=1}^n\int_v^{v_j}\!\!\y_j(\xi_j,\ol z)^{\rm t}\,\y_j(\xi_j,z)\,\d\xi_j.
\end{equation}
It remains to be shown that the integral in \eqref{mjdot} equals $\dot m_j(z)$. To this end, 
we write $m_j(z)=\dfrac{b_j(z)}{c_j(z)}$  as in~\eqref{c}. 
Then, using Lemma \ref{intyy}, \eqref{m1}, and \eqref{m}, we find 
\begin{align*}
\int_v^{v_j}\!\!\y_j(\xi_j,\ol z)^{\rm t}\,&\y_j(\xi_j,z)\,\d\xi_j
 =\big( \!\!-\!\!1\,  m_j(z)\big)\int_v^{v_j}!\!Y_j(\xi_j,\ol z)^*\,Y_j(\xi_j,z)\,\d\xi_j\,\begin{pmatrix}-1\\m_j(z)\end{pmatrix}\\[-1mm]
&=-\big( -1\ m_j(z)\big)\,Y_j(v_j,\ol z)^*\,J \dot Y_j(v_j,z)\,\begin{pmatrix}-1\\m_j(z)\end{pmatrix}\\
&=-\dfrac 1{c_j(z)}\big( -\sin\al_j\ \cos\al_j\big)J\dot Y_j(v_j,z)\,\begin{pmatrix}-1\\m_j(z)\end{pmatrix}\\
&=-\dfrac 1{c_j(z)}\big(\cos\al_j\ \sin\al_j\big)\begin{pmatrix}\dot y_{j,11}(v_j,z)&\dot y_{j,12}(v_j,z)\\
   \dot y_{j,21}(v_j,z)&\dot y_{j,22}(v_j,z)\end{pmatrix}\begin{pmatrix}-1\\ \frac{b_j(z)}{c_j(z)}\end{pmatrix}\\
&=-\dfrac 1{c_j(z)}\big(\,\dot b_j(z)\  \dot c_j(z)\big)\begin{pmatrix}-1\\\frac{b_j(z)}{c_j(z)}\end{pmatrix}\\
&= - \dfrac{-\dot b_j(z)c_j(z)+b_j(z)\dot c_j(z)}{c_j(z)^2}\\
&= \dot m_j(z).
\qedhere
\end{align*}
\end{proof}

\vspace{2mm}

\section{Dislocation of eigenvalues for Robin matching conditions}
\label{sec:disloc}

In this section we study the difference $d_R(\mT_\tau)$ of the number of eigenvalues in intervals $[0,R)$ and $[-R,0)$ for the Dirac-Krein operator $\mT_\tau$. 
We prove that there exists a number $\kappa_0\in\Z$ such that, for $R>0$ sufficiently large,  $d_R(\mT_\tau)$ differs from $\kappa_0$ at most by $n+2$; 
in fact, we even show a stronger result in terms of the boundary parameters $\al_j$ at the outer vertices.

\medskip

\subsection{}
As a first tool, we invoke results on the asymptotics of the eigenvalues of Dirac-Krein operators on a compact interval with summable potentials (see e.g.\ \cite{MR2201307}, \cite{LM14}).
In view of the boundary conditions \eqref{bc} imposed for the operators $T_j$ on an edge $e_j=[v,v_j]$ for fixed $j\in\{1,2,\dots,n\}$, we need to apply the more general results in \cite{LM14} \ct{and the detailed version \cite{MR3340175} (see also \cite{MR3248475})}. To this end, we first have to transform the differential equation \eqref{eq} into the form considered in \cite{LM14} by means of the unitary matrix $W$ diagonalizing~$J$,
\begin{equation}
 \label{W}
   W:=\frac{1}{\sqrt{2}}\begin{pmatrix}1&1\\-\i&\i\end{pmatrix}, \quad W^*W=I_2, \quad W^*JW=\begin{pmatrix}\i&0\\0&-\i\end{pmatrix}.
\end{equation}

\begin{prop}
\label{trafo}
A function $\f_j\!=\!\big(f_j\, \wh f_j\big)^{\rm t} \in \L^2(e_j) \oplus \L^2(e_j)$ is a solution of the 
homogeneous Dirac-Krein system \eqref{eq} for $z\in \C$ with boundary conditions~\eqref{bc} if and only if 
the function $\h_j:=  W^* \f_j \!=\! \big(h_j\, \wh h_j\big)^{\rm t} \in \L^2(e_j) \oplus \L^2(e_j)$ is a solution of
\begin{align}
\label{trafo-eq}
  & \begin{pmatrix}\i\!\!& \,0\\0\!\!&\!-\i\end{pmatrix} \h_j'(x) \!-\!  
  \begin{pmatrix}0 \!&\! r_j(x)  \\ \overline{r_j(x)} \!&\! 0\end{pmatrix} 
  \h_j(x) - \la \h_j(x) 
  \!=\! 0, 
  \quad x\!\in\! e_j, \\
\label{trafo-bc}
  & \ h_j(v)+\wh h_j(v)=0, \quad \e^{-\i\al_j} h_j(v_j) + \e^{\i\al_j} \wh h_j(v_j) = 0,
\end{align}
with $\la=-z$ where $r_j:=p_j+\i q_j \in L^1(e_j,\C)$.
\end{prop}

\begin{proof}
Let $x\in e_j$ and $z\in\C$. If we insert $\f_j =  W \h_j$ into \eqref{eq} and multiply the resulting equation  by $-W^*$ from the left, we obtain that \eqref{eq} is equivalent to 
\[
  W^*JW \h_j'(x) - W^* V_j(x) W \h_j(x) + z \h_j(x) = 0. 
\]
The latter coincides with \eqref{trafo-eq} because of \eqref{W} and 
\[
 W^*V_j(x){W}=\begin{pmatrix}0 & p_j(x)+\i q_j(x) \\ p_j(x)-\i q_j(x) & 0\end{pmatrix}\!, \quad x\in e_j.
\] 
The equivalence of \eqref{bc}, \eqref{trafo-bc} is also immediate from $\f_j \!=\! W \h_j$.
\end{proof}

In the following, let $l_j:=v_j-v$ be the length of the edge $e_j=[v,v_j]$. For the problem \eqref{eq}, \eqref{bc} with $V_j\equiv 0$, we denote the corresponding operator, 
fundamental matrix, eigenvalues, and denominator of the Weyl--Titchmarsh function by $T_j^0$, $Y_j^0(\cdot,z)$, $\mu_{j,k}^0$, $k\in\Z$, and $c_j^0(z)$, i.e.\
\begin{align}
\nonumber
& Y_j^0(x,z) = \begin{pmatrix} \cos((x-v)z)  & - \sin ((x-v)z) \\ \sin ((x-v)z) & \ \ \cos((x-v)z) \end{pmatrix},\\
\label{ev-unpert}
& \mu_{j,k}^0 = \frac{ \alpha_j +  k \pi }{l_j}, \quad \quad k \in \Z, \\
\label{cj0def}
& c_j^0(z) = - \sin((v_j-v)z) \cos \alpha_j + \cos ((v_j-v)z) \sin \alpha_j = \sin (\alpha_j - l_j z).
\end{align} 


\begin{prop}
\label{trafo-asymp}
Let $p_j$, $q_j \in L^1(e_j)$ and $\al_j\in[0, \pi)$. Then the fundamental matrix $Y_j(\cdot,z)$ of \eqref{eq} with $Y_j(v,z)=I_2$ $($see \eqref{Y}$)$ 
and the eigenvalues $\mu_{j,k}$, $k\!\in\!\Z$, of the operator $T_j$, suitably enumerated, admit the asymptotic expansions
\begin{alignat}{2}
\label{fund-asymp}
   Y_j(x,z) &=  Y_j^0(x,z) + {\rm o}_{\!j}(1), \quad & & |z| \to \infty, \qquad \text{uniformly in $x\in e_j$},\\
\label{ev-asymp}
  \mu_{j,k} &= \mu_{j,k}^0 + {\rm o}_{\!j}(1), \quad & & |k| \to \infty;\\
\intertext{here the subscript $j$ indicates that the lower order terms depend on $j$.~Further,}
\label{cj0}
 c_j(z) &= c_j^0(z) +  {\rm o}_{\!j}(1), \quad  & &  |z| \to \infty.
\end{alignat}
\end{prop}

\begin{proof}
If we transform the independent variable $x\in e_j=[v,v_j]$ to $t\in [0,1]$ by means of $x=v+l_j t$, we find that the system \eqref{trafo-eq} is a special case of the more general Dirac-type system \cite[(1)]{LM14} with
$b_1=-l_j$, $b_2=l_j$, $Q_{12}(t) = p_j(v+l_j t)$ $+\i q_j(v+l_j t)$, $Q_{21}=p_j(v+l_j t)-\i q_j(v+l_j t)$, $t\in[0,1]$. Further, the boundary con\-ditions \eqref{trafo-bc} are the special case 
$a_{11}\!=\!a_{12}\!=1$, $a_{13}\!=\!a_{14}\!=\!0$ and $a_{21}\!=\!a_{22}\!=\!0$, $a_{23}\!=\!\e^{-\i\al_j}$, $a_{24}\!=\!\e^{\i\al_j}$ of the boundary conditions \cite[(3)]{LM14}; in particular, the boundary conditions are separated and hence strictly regular 
(see \cite[Rem.~1]{LM14}). Moreover, the assumption $p_j$, $q_j \in L^1(e_j)$ implies that $Q_{12}$, $Q_{21} \in L^1(0,1)$ and hence all assumptions of \cite[Prop.~1-3]{LM14} are satisfied.

By Proposition \ref{trafo}, $Y_j(x,z)$, $x\!\in\! e_j$, is the fundamental matrix of the Di\-rac-Krein system \eqref{eq} with $Y_j(v,z)\!=\!I_2$ 
if and only if $W^* Y_j(v+l_j t,-\la) W \!=\! \Phi_j(t,\la)$, $t\in[0,1]$, is the fundamental matrix of the corresponding transformed system \cite[(1)]{LM14} with $\Phi_j(0,\la)=I_2$. By \cite[Prop.~2]{LM14}, we~have
\[
 \Phi_j(t,\la) = \begin{pmatrix} \e^{- \i l_j t \la} & 0 \\ 0 & \e^{\i l_j t \la} \end{pmatrix} + {\rm o}_{\!j}(1), \quad |\la| \to \infty,
\]
uniformly in $t\in [0,1]$, 
which implies the asymptotic expansion in \eqref{fund-asymp} if we note that $l_j t = x-v$. Further, it is not difficult to check that
the unperturbed eigenvalues $\la_k^0$ in \cite[(32)]{LM14} can be computed explicitly to~be (note that $b_1-b_2=-2l_j$  therein)  
$\la_k^0 = -\mu_{j,k}^0$, 
$k \in \Z$. Now the asymptotic expansion \eqref{ev-asymp} follows from \cite[(33)]{LM14}.

Finally, \eqref{cj0} is immediate from \eqref{fund-asymp} and the definition of $c_j$, $c_j^0$ (see \eqref{c}, \eqref{cj0def}, respectively).
\end{proof}

\subsection{}
Now we return to the Dirac-Krein operator $\mT_\tau$ on the star graph $\cG$ with potentials $V_j$ on the edges $e_j$, 
boundary conditions \eqref{bc_j} at the outer vertices $v_j$, $j=1,2,\dots,n$, and Robin matching conditions \eqref{k1}, \eqref{k2} at the central 
vertex~$v$.

For a self-adjoint operator $T$ with discrete spectrum, in particular for non-semi-bounded $T$, we introduce the following eigenvalue counting functions.
If $\la_1$, $\la_2 \in \R$, $\la_1 <\la_2$, and $\la \in [0,\infty)$, $\mu\in (-\infty,0)$, we set
\begin{gather*}
N\big([\la_1,\la_2);T\big) := \# \big( \sigma(T) \cap [\la_1,\la_2) \big),\\
N_+(\la;T):= N( [0,\la);T \big), 
\quad
N_-(\mu;T):= N( [\mu,0);T \big),  
\end{gather*}
and, if $R>0$,
\begin{equation}
\label{dR}
  d_R(T):= N_+(R;T) - N_-(-R;T). \quad
\end{equation}

\smallskip

The next theorem is the main result of this section. It shows that, for all sufficiently large $R>0$, the difference 
$d_R(\mT_\tau)$  of the numbers of eigenvalues of $\mT_\tau$ in $[0,R)$ and $[-R,0)$ deviates from some fixed integer $\kappa_0$, which we call \emph{dislocation index}, at most by $n+2$. In fact, the deviation can be expressed even more precisely in terms of the boundary conditions at the outer vertices.

\begin{thm}
\label{finally!}
There exists a number $\kappa_0 \!\in\! \Z$ such that, for $R\!>\!0$ sufficiently~large,
\begin{equation}
\label{eq:final}
 -(n_\ge + 2 ) \le d_R(\mT_\tau) - \kappa_0 \le n_\le + 2
\end{equation}
with
\begin{align*}
n_\ge &:= \# \Big\{ j\in \{1,2,\dots,n\} : \alpha_j \ge \frac \pi2\Big\},\quad
n_\le := \# \Big\{ j\in \{1,2,\dots,n\} : \alpha_j \le \frac \pi2\Big\},
\end{align*}
where $\al_j$ are the parameters in the boundary conditions \eqref{bc_j} at the outer vertices. 
\end{thm}

\begin{rem}
It is immediate from \eqref{eq:final} that, for $R>0$ sufficiently~large,
\[
  | d_R(\mT_\tau) - \kappa_0 | \le n+2.
\]
\end{rem}

\smallskip

For the proof of Theorem \ref{finally!}, and to obtain a formula for the dislocation index~$\kappa_0$, we need some auxiliary results for the Dirac-Krein operators $T_j$ on the edges $e_j$ for fixed $j\in \{1,2,\dots,n\}$.

By \eqref{ev-unpert}, the sequence of eigenvalues $(\mu_{j,k}^0)_{k\in\Z}$ of the operator $T_j^0$ with $V_j \equiv 0$ is given by $\mu_{j,k}^0=\frac{ \alpha_j +  k \pi }{l_j}$, $k\in\Z$, and ordered such that
\[
\cdots \le \mu_{j,-2}^0 \le \mu_{j,-1}^0 < 0 \le \mu_{j,0}^0 \le \mu_{j,1}^0 \le \mu_{j,2}^0 \le \cdots.
\]
Correspondingly, the sequence of eigenvalues $(\mu_{j,k})_{k\in\Z}$ of $T_j$ is enumerated such that $\mu_{j,k} = \mu_{j,k}^0 + {\rm o}_{\!j}(1)$, $|k| \to \infty$ (see \eqref{ev-asymp}).

\begin{lem}
\label{dislocj}
For every  $\delta\in(0,\frac \pi{2l_j})$, there exists a $K_\delta \in \N$ such that  
\begin{align}
\label{etaj}
   \kappa(T_j) &:= 
   N_{\!+} \!\left(\mu_{j ,k-1}^0+ \delta;T_j \right) 
   - N_{\!-} \!\left(\mu_{j ,-k}^0 - \delta;T_j  \right) 
\end{align}
is constant for all $k\!\ge\!K_\delta$, independent of $\delta$, and even; in particular, $\kappa(T_j^0)\!=\!0$.
\end{lem}

\begin{proof}
Let $\delta\in(0,\frac \pi{2l_j})$ be arbitrary. 
It is not difficult to see that, due to the eigenvalue asymptotics \eqref{ev-asymp}, 
there exists $K_{\delta}\in \mathbb{N}_{0}$ such that, for all $k\in\Z$, $k_1$, $k_2\in\Z$, $k_1 \le k_2$, with $|k|, |k_1|, |k_2| \ge K_\delta$, we have $\mu_{j,k}^0\pm \delta \notin \sigma(T_j)$~and
\begin{align}
\label{N1}
&N \left( \big[ \mu_{j,k}^0 - \delta, \mu_{j,k}^0 + \delta \big); \, T_j \right) = 1;\\
\label{N2}
&N \left( \big[ \mu_{j,k}^0 + \delta, \mu_{j,k+1}^0 - \delta  \big) ; \, T_j \right) = 0;\\
\label{N3}
&N \left( \big[  \mu_{j,k_1}^0 \!-\! \delta, \mu_{j,k_2-1}^0 \!+\! \delta \big); \, T_j \right) = k_2-k_1.
\end{align}
Then, for all $k\ge K_\delta$, by \eqref{N2} and \eqref{N1}, 
\begin{align*}
   N_{\!+} \!\left(\mu_{j ,k-1}^0+ \delta;T_j   \right) 
   &= N_{\!+} \!\left(\mu_{j ,K_\delta-1}^0+ \delta;T_j   \right)  + N \big( \big[\mu_{j ,K_\delta}^0 -\delta,  \mu_{j ,k-1}^0+ \delta\big); T_j \big)\\
   & = N_{\!+} \!\left(\mu_{j ,K_\delta-1}^0+ \delta;T_j   \right) + k - K_\delta, \\[-2mm]
\intertext{and, analogously,}
   N_{\!-} \!\left(\mu_{j ,-k}^0 - \delta;T_j   \right) 
   & =  N_{\!-} \!\left(\mu_{j ,-K_\delta}^0 - \delta;T_j   \right) + k - K_\delta.
\end{align*}%
Hence for $k\ge K_\delta$, the right hand side of \eqref{etaj}
\begin{align}
   \label{withdelta}
   \begin{aligned}
      \kappa(T_j,\delta) 
      &:=
      N_{\!+} \!\left(\mu_{j ,k-1}^0\!+ \delta;T_j   \right) - N_{\!-} \!\left(\mu_{j ,-k}^0 \!- \delta;T_j   \right)
      \\
      &= 
      N_{\!+} \!\left(\mu_{j ,K_\delta-1}^0\!+ \delta;T_j   \right) - N_{\!-} \!\left(\mu_{j ,-K_\delta}^0 \!- \delta;T_j   \right)
   \end{aligned}
\end{align}
is independent of $k$.
The sum of the two terms on the right hand side is even because, by~\eqref{N3},
\begin{align}
\nonumber
\hspace{-3mm} N_{\!+} \!\big(\mu_{j ,K_\delta-1}^0 \!+ \delta;T_j \big) +  N_{\!-} \!\big(\mu_{j ,-K_\delta}^0 \!- \delta;T_j \big) 
& = N\big( \big[ \mu_{j ,-K_\delta}^0 \!- \delta,\mu_{j ,K_\delta-1}^0 \!+ \delta \big); T_j \big) \hspace{-7mm}\\[1mm]
& = 2 K_\delta \in 2 \Z; \label{sumwithdelta}
\end{align}
therefore the difference must be even as well, i.e.\ $\kappa(T_j,\delta) \in 2 \Z$.

It remains to be shown that $\kappa(T_j,\delta)$ does not depend on $\delta$ and hence $\kappa(T_j)$ in \eqref{etaj} is well-defined.
To this end, let $\delta_1$, $\delta_2\in(0,\frac \pi{2l_j})$, without loss of generality $\delta_1<\delta_2$. Then
we may choose $K_{\delta_2} \le K_{\delta_1}$ and, by \eqref{N1} and \eqref{N2}, we obtain
\begin{align*}
  &N_{\!+} \!\big(\mu_{j ,K_{\delta_1}\!-1}^0 \!+ \delta_1;T_j   \big) \\
  &= N_{\!+} \!\big(\mu_{j ,K_{\delta_2}\!-1}^0 \!+ \delta_2;T_j   \big) + N\big( \big[ \mu_{j ,K_{\delta_2}-1}^0 \!+ \delta_2, \mu_{j ,K_{\delta_1}\!-1}^0 \!+ \delta_1\big);T_j  \big) \\
  &= N_{\!+} \!\big(\mu_{j ,K_{\delta_2}\!-1}^0 \!+ \delta_2;T_j   \big) + N\big( \big[ \mu_{j ,K_{\delta_2}}^0 \!- \delta_2, \mu_{j ,K_{\delta_1}\!-1}^0 \!+ \delta_1\big);T_j  \big) \\
  &= N_{\!+} \!\big(\mu_{j ,K_{\delta_2}\!-1}^0 \!+ \delta_2;T_j   \big) + K_{\delta_1} - K_{\delta_2}, 
\end{align*}  
and, analogously,
\[
  \hspace*{-8mm}N_{\!-} \!\big(\mu_{j ,-K_{\delta_1}}^0 \! -\delta_1;T_j  \big) 
   = N_{\!-} \!\big(\mu_{j ,-K_{\delta_2}}^0 \! - \delta_2;T_j   \big) + K_{\delta_1} - K_{\delta_2}.
\]
Taking the difference and using \eqref{withdelta}, we conclude $\kappa(T_j,\delta_1)=\kappa(T_j,\delta_2)$.
\end{proof}

\begin{lem} 
\label{lastlemma}
There exists an $R_j >0$ such that, for all $R>R_j$,
\begin{equation}
\label{jazz}
 d_R(T_j) - \kappa(T_j) \in \ \begin{cases}
                              \{0,1\} & \text{ if } \ \al_j < \frac \pi 2,\\
                              \{-1,0,1\} & \text{ if } \ \al_j = \frac \pi 2,\\
                              \{-1,0\} & \text{ if } \ \al_j > \frac \pi 2.
                              \end{cases}
\end{equation}
\end{lem}

\begin{proof}
Let $\delta>0$ be such that $\delta < \dfrac{|\frac \pi 2 - \al_j|}{l_j}$. Then $\delta < \frac \pi{2l_j}$ and we choose $K_\delta \in \N$ according to Lemma \ref{dislocj}.
If $\alpha_j < \frac \pi 2$, then  
\begin{align*}
  \mu_{j,-k}^0 + \delta - \big( - (\mu_{j,k-1}^0 - \delta)\big) = \mu_{j,-k}^0 + \mu_{j,k-1}^0 + 2 \delta = \frac {2\al_j-\pi}{l_j} + 2 \delta < 0
\end{align*}
for $k\!\in\!\N$. Hence the interval $(\mu_{j,k-1}^0\!-\!\delta, \mu_{j,k-1}^0\!+\!\delta)$, reflected at $0$, lies to~the right of the interval $(\mu_{j,-k}^0\!-\!\delta, \mu_{j,-k}^0\!+\!\delta)$ 
and thus, by \eqref{dR} and \eqref{N1},~\eqref{N2},
\begin{align*}
 \kappa(T_j) &\!=\! N_+(\mu_{j,k-1}^0\!+\!\delta;T_j) \!-\! \big( N_-(-(\mu_{j,k-1}^0\!+\!\delta);T_j) \\
 & \ \ \!+\! N\big([\mu_{j,-k}^0\!-\!\delta, -(\mu_{j,k-1}^0\!+\!\delta)); T_j) \big)\\
 &\!=\! d_{\mu_{j,k-1}^0\!+\!\delta}(T_j) -  1. 
\end{align*}
This proves \eqref{jazz} for $R=R_j:=\mu_{j,k-1}^0\!+\!\delta$ if $\alpha_j < \frac \pi 2$. For $R\ge R_j$, the claim follows if we note that
\begin{align*}
 d_R(T_j) - \kappa(T_j) &= d_R(T_j) - d_{\mu_{j,k-1}^0\!+\!\delta}(T_j) - \big( \kappa(T_j) - d_{\mu_{j,k-1}^0\!+\!\delta}(T_j) \big) \\
  &= d_R(T_j) - d_{\mu_{j,k-1}^0\!+\!\delta}(T_j) - (-1)
\end{align*}
and, since  $\alpha_j < \frac \pi 2$,
\[
  d_R(T_j) - d_{\mu_{j,k-1}^0\!+\!\delta}(T_j) \in \{-1,0\}.
\]
The proof for the case $\al_j > \frac \pi 2$ is similar. For $\al_j = \frac \pi 2$ we have
$-\mu_{j,k-1}^0 =\mu_{j,-k}^0$, hence $\kappa(T_j) = d_{\mu_{j,k-1}^0\!+\!\delta}(T_j)$. Now the claim follows since, for $R \ge R_j$,
\begin{align*}
  & d_R(T_j) - d_{\mu_{j,k-1}^0\!+\!\delta}(T_j) \in \{-1,0,1\}. \qedhere
\end{align*}
\end{proof}

\noindent
\emph{Proof of Theorem} \ref{finally!}. 
We show that the two-sided inequality \eqref{eq:final} holds~with
\begin{equation}
\label{kappa_tau}
  \kappa_0 = \sum_{j=1}^n \kappa(T_j),
\end{equation}
where $\kappa(T_j)$ is defined as in \eqref{etaj}.
The decoupled operator $\mT_0$, corresponding to Dirichlet conditions at the central vertex $v$, and its eigenvalues are \vspace{-2mm} given~by 
\[
 \mT_0 = \bigoplus_{j=1}^n T_j, \quad \sigma(\mT_0) = \bigcup_{j=1}^n \big\{ \mu_{j,k}^0: k\in \Z \big\},
\]
hence $d_R(\mT_0)= \sum_{j=1}^n d_R(T_j)$. By \eqref{jazz}, we can estimate
\[
  d_R(T_j) - \kappa(T_j)  \le \begin{cases} 1 & \text{ if } \al_j \le \frac \pi2,\\
                                             0 & \text{ if } \al_j > \frac \pi2.
                               \end{cases}               
\]
Therefore, together with \eqref{kappa_tau},  we obtain 
\begin{equation}
\label{last?}
 d_R(\mT_0) - \kappa_0 = \sum_{j=1}^n \big( d_R(T_j) - \kappa(T_j) \big) \le n_\le.  
\end{equation}
Further, by Theorem \ref{thmTK}, the eigenvalues of $\mT_\tau$ and $\mT_0$ interlace. From this we conclude, considering the eigenvalues near $-R$, $R$, and $0$, that 
\[
  |d_R(\mT_\tau) - d_R(\mT_0)| \le 2, 
\]
which, together with \eqref{last?}, proves the upper bound in \eqref{eq:final}. The proof for the lower bound is analogous.
\qed

\vspace{2mm}

\subsection{}
In this final subsection we derive an analytic formula for the numbers $\kappa(T_j)$, and hence for the dislocation index $\kappa_0$ of $\mT_\tau$ in Theorem \ref{finally!} given by \eqref{kappa_tau}. 
It is based on the trace formula for $T_j$ proved in Section~\ref{sec:CanonicalSystems}.

\begin{prop}
For $z\in \rho(T_j)$, let $Y_j(\cdot,z)=\left(y_{j,kl}(\cdot,z)\right)_{k,l=1}^2$ and let 
\[
  c_j(z)= y_{j,12}(v_j,z)\cos\al_j+y_{j,22}(v_j,z)\sin\al_j
\]   
be the denominator in the Weyl--Titchmarsh function $m_j$ in \eqref{c}. Further,~set
\[
  \omega_{j,-} := \max \big( \sigma(T_j) \cap (-\infty,0) \big), \quad
  \omega_{j,+} := \min \big( \sigma(T_j) \cap [0,\infty) \big).
\]%
Then, for every $\omega_j \in (\omega_{j,-},\omega_{j,+})$,
\begin{equation}
   \label{eq:kappaformula}
   \begin{aligned}
      \kappa(T_j)
      &=
      -\frac{1}{\pi}
      \int_\R'
      \frac{ \dot c_j( \omega_j + \i s) }{ c_j( \omega_j + \i s) }\d s
      \\
      &= \frac{1}{\pi} \lim_{M\to\infty} \Big(
      \arg \Big( c_j  \Big( \omega_j - \i M \Big) \Big)
      - \arg \Big( c_j \Big( \omega_j + \i M \Big) \Big) \Big),
   \end{aligned}
\end{equation}
where ${\phantom I\!\!}^\prime$ denotes the Cauchy principal value at $\infty$.
\end{prop}

\begin{proof}
By the residue theorem, for $-\infty < \lambda_1 < \lambda_2 < \infty$ and $\mu\in\mathbb R$,
\begin{align*}
   \frac{1}{2\pi} 
   \int_\R' 
   \Big(
   \frac{ 1 }{ \mu - \lambda_1 - \i s} -
   \frac{ 1 }{ \mu - \lambda_2 - \i s} \Big) \d s
   =
   \begin{cases}
      1,\ & \mu\in(\lambda_1,\,\lambda_2),
      \\
      0,\ & \mu\notin [\lambda_1,\,\lambda_2].
   \end{cases}
\end{align*}
This and Proposition \ref{traceformula1} imply that, if $\lambda_1$, $\lambda_2\notin\sigma(T_j)$, then the number of eigenvalues of $T_j$ in $(\lambda_1,\lambda_2)$, counted with multiplicity, is
\begin{align}
\nonumber
   N([\lambda_1,\lambda_2);T_j) 
   &=
   \frac{1}{2\pi} 
   \sum_{k\in\mathbb Z}
   \int_\R'  \Big(
   \frac{ 1 }{ \mu_{j,k} - \lambda_1 - \i s} -
   \frac{ 1 }{ \mu_{j,k} - \lambda_2 - \i s} \Big) \d s
   \\
\label{catastrophe}   
   &= 
   \frac{1}{2\pi} 
   \int_\R' 
   \tr\,\big( (T_j - \lambda_1 -\i s)^{-1} - (T_j - \lambda_2 -\i s)^{-1} \big)
   \d s
   \\
\nonumber   
   &= 
   \frac{1}{2\pi} 
   \int_\R' \Big(
   \frac{ \dot c_j(\lambda_2 + \i s) }{ c_j(\lambda_2 + \i s) }
   -\frac{ \dot c_j(\lambda_1 + \i s) }{ c_j(\lambda_1 + \i s) }\Big) 
   \d s.
\end{align}
By Lemma \ref{dislocj}, for every $\delta\in (0,\frac \pi{2l_j})$ there exists $K_\delta\in\N$ such that for all $k\in\N_0$, $k \ge K_\delta$,
we have $\mu_{j,k}^0\pm \delta \notin \sigma(T_j)$ \vspace{-1mm}  and
\begin{equation}
\label{eq:kappa}
\kappa(T_j) = N_+\big( \mu_{j,k-1}^0\!+\delta;T_j\big) - N_-\big( \mu_{j,-k}^0\!-\delta;T_j\big).
\end{equation}
Since $(\omega_{j,-}, \omega_{j,+}) \cap \sigma(T_j)  = \emptyset$ and $\omega_j \in (\omega_{j,-}, \omega_{j,+})$, it follows that
\[
\kappa(T_j) = N\big( [\omega_j,\mu_{j,k-1}^0\!+\delta);T_j\big) - N\big( [\mu_{j,-k}^0\!-\delta,\omega_j);T_j\big).
\]
Then, due to \eqref{catastrophe},
\begin{align*}
\kappa(T_j) = \frac 1{2\pi} \int_\R' \Big( 
\frac{ \dot c_j(\mu_{j,k-1}^0\!+\delta + \i s) }{ c_j(\mu_{j,k-1}^0\!+\delta + \i s) } +
\frac{ \dot c_j(\mu_{j,-k}^0\!-\delta + \i s) }{ c_j(\mu_{j,-k}^0\!-\delta + \i s) } -
2 \frac{ \dot c_j(\omega_j + \i s) }{ c_j(\omega_j + \i s) }
\Big) \d s.
\end{align*}
It is not difficult to check that $\dfrac{\dot c_j}{c_j} = \dfrac{\dot c_j^0}{c_j^0} +  \dfrac{\big(\frac{c_j}{c_j^0}\big)^{\bf{\!\cdot}}}{\frac{c_j}{c_j^0}}$. \vspace{-2mm}
Then
\begin{align*}
\kappa(T_j)  = & \frac 1{2\pi} \int_\R' \Big( 
\frac{ \dot c_j^0(\mu_{j,k-1}^0\!+\delta + \i s) }{ c_j^0(\mu_{j,k-1}^0\!+\delta + \i s) } +
\frac{ \dot c_j^0(\mu_{j,-k}^0\!-\delta + \i s) }{ c_j^0(\mu_{j,-k}^0\!-\delta + \i s) } \Big) \d s\\
& + \frac 1{2\pi} \int_\R' \Big( 
\frac{ \big(\frac{c_j}{c_j^0}\big)^{\bf{\!\cdot}}(\mu_{j,k-1}^0\!+\delta + \i s) }{ \frac{c_j}{c_j^0}(\mu_{j,k-1}^0\!+\delta + \i s) } +
\frac{ \big(\frac{c_j}{c_j^0}\big)^{\bf{\!\cdot}}(\mu_{j,-k}^0\!-\delta + \i s) }{ \frac{c_j}{c_j^0}(\mu_{j,-k}^0\!-\delta + \i s) } \Big) \d s
\\
& - \frac 1{2\pi} \int_\R' 2 \frac{ \dot c_j(\omega_j + \i s) }{ c_j(\omega_j + \i s) } \d s.
\end{align*}
We denote  the first two integrals in the above formula by $I_{1,k}$ and $I_{2,k}$ and show that $I_{1,k} = 0$ for $k\ge K_\delta$ and $I_{2,k} = 0$ for sufficiently large $k\ge K_\delta$;
then the first equality in \eqref{eq:kappaformula}~follows.
By means of \eqref{catastrophe}, we conclude that, for $k \ge K_\delta$,
\begin{align*}
I_{1,k} =&\, N\Big( \Big[ \frac{\alpha_j\!-\!\frac\pi2}{l_j}, \mu_{j,k-1}^0\!\!+\!\delta\Big); T_j^0 \Big) 
        - N \Big( \Big[ \mu_{j,-k}^0\!\!-\!\delta , \frac{\alpha_j\!-\!\frac\pi2}{l_j}\Big); T_j^0 \Big)\\
       &+\! \frac 1 \pi \!\int_\R' \!\frac{ \dot c_j^0( \frac{\alpha_j\!-\!\frac\pi2}{l_j}\!+\!\i s) }{ c_j^0( \frac{\alpha_j\!-\!\frac\pi2}{l_j}\!+\! \i s) } \d s.
\end{align*}
Since the first two numbers are both equal to $k$ and since $c_j^0(z)=\sin(\alpha_j-l_j z)$ by \eqref{cj0}, we obtain
\begin{align*}
I_{1,k} = - l_j  \frac 1 \pi \int_\R'\cot \Big( \frac \pi2 + \i s \Big) \d s = \i\, l_j \frac 1 \pi \int_\R' \tanh s \,\d s = 0, \quad k\ge K_\delta.
\end{align*}
Because $Y_j(\cdot,\ol{z})= \ol{Y_j(\cdot,z)}$, we have $\frac{c_j}{c_j^0}(\ol{z}) = \ol{\frac{c_j}{c_j^0}(z)}$ and hence
\begin{align*}
I_{2,k}= & \frac 1{2\pi} \lim_{M\to\infty} \!\!\bigg( \int_{\mu_{j,k-1}^0\!+\delta - \i M}^{\mu_{j,k-1}^0\!+\delta + \i M} \!\!\!
(-\i) \,\d_\theta \Big( \ln \Big| \frac{c_j}{c_j^0}(\theta)\Big| + \i \arg \Big( \frac{c_j}{c_j^0}(\theta) \Big) \Big)  \\
& \hspace{14mm} + \int_{\mu_{j,-k}^0\!-\delta - \i M}^{\mu_{j,-k}^0\!-\delta + \i M}  \!\!\!\!
(-\i) \,\d_\theta \Big( \ln \Big|\frac{c_j}{c_j^0}(\theta)\Big| + \i \arg \Big(\frac{c_j}{c_j^0}(\theta)\Big) \Big) \bigg)\\
=& \frac 1{2\pi} \lim_{M\to\infty} \!\! \bigg( \int_{\mu_{j,k-1}^0\!+\delta - \i M}^{\mu_{j,k-1}^0\!+\delta + \i M} \!\!\!
\d_\theta \arg \Big(\frac{c_j}{c_j^0}(\theta)\Big) + \int_{\mu_{j,-k}^0\!-\delta - \i M}^{\mu_{j,-k}^0\!-\delta + \i M} \!\!\!
\d_\theta \arg \Big(\frac{c_j}{c_j^0}(\theta)\Big) \bigg) \\
=& \frac 1{2\pi} \lim_{M\to\infty} \!\! \bigg( \int_{\mu_{j,k-1}^0\!+\delta - \i M}^{\mu_{j,k-1}^0\!+\delta + \i M} \!\!\!
\d_\theta \arg \Big( 1 + \frac{c_j-c_j^0}{c_j^0}(\theta)\Big)  \\
& \hspace{14.5mm} + \! \int_{\mu_{j,-k}^0\!-\delta - \i M}^{\mu_{j,-k}^0\!-\delta + \i M} \!\!\!
\d_\theta  \arg \Big(1 + \frac{c_j-c_j^0}{c_j^0}(\theta)\Big) \bigg). 
\end{align*}
We denote the last two integrals by $I_{2,k}^+$ and $I_{2,k}^-$ and show that, for sufficiently large $k\ge K_\delta$, they tend to $0$ for $M\to\infty$.
If we use that $\mu_{j,k-1}^0$ is a zero of $c_j^0(z)= \sin(\alpha_j-l_jz)$, it is elementary to show that, 
for all $\theta= \mu_{j,k-1}^0\!+\!\delta \!+\! \i t$, $t \!\in\! \R$,
\[
  |c_j^0(\theta)| \ge |\sin (l_j \delta) | \cosh t \ge   |\sin (l_j \delta) |>0. 
\]
Since $(c_j-c_j^0)(z)={\rm o}_j(z)$, $|z|\to \infty$, by \eqref{cj0}, we have
\begin{equation}
\label{end1}
  \lim_{M\to\infty}  \frac{c_j-c_j^0}{c_j^0}(\mu_{j,k-1}^0 +\delta \pm \i M) = 0.
\end{equation}
Moreover, we can choose $k_0\ge K_\delta$ so large that for all $\theta= \mu_{j,k_0-1}^0+\delta + \i t$, $t \in \R$,
we can estimate $|(c_j-c_j^0)(\theta)| < \sin(l_j \delta)/2$ and \vspace{-1mm} hence  
\begin{equation}
\label{end2}
  \Big| \frac{c_j-c_j^0}{c_j^0}(\theta) \Big| < \frac 12.
\end{equation}
Now \eqref{end1} and \eqref{end2} imply that $I_{2,k_0}^+=0$. Analogously, we obtain  $I_{2,k_0}^-=0$ and hence $I_{2,k_0} = 0$.

The second equality in \eqref{eq:kappaformula} follows from the first one using the symmetry property of $\frac{\dot c_j}{c_j}$ in the same way as in the reasoning for $I_{2,k}$ above.
\end{proof}

\subsection*{Acknowledgment}

This work was supported by the Swiss National Science Foundation, SNF Scopes Grant\,IZ73Z0\-\_128135. C.\,Tretter also gratefully acknowledges a guest professorship of the 
\emph{Knut\,och\,Alice\,Wallenbergs~Stiftelse} and thanks the Mate\-matiska institutionen at Stockholms universitet for the kind~hospitality.

\bibliographystyle{alpha-abbrv}
\bibliography{2015-05-18}
\vspace{1mm}

\end{document}